\documentclass[a4paper,10pt,twoside]{article}
\usepackage[utf8]{inputenc}
\usepackage[greek,english]{babel}
\languageattribute{greek}{polutoniko}

\usepackage{amsmath, enumerate} 
\usepackage{mathrsfs}  
\usepackage{amsthm}
\usepackage{amsfonts}
\usepackage{amssymb, latexsym}
\usepackage{mhequ}
\usepackage{verbatim}

\usepackage{kerkis}

\usepackage{graphicx}
\usepackage{color}
\usepackage{xcolor}
\usepackage{tikz}
\usetikzlibrary{calc,intersections,through,backgrounds}
\usetikzlibrary{decorations.pathreplacing}
\usetikzlibrary{shapes}

\usepackage{float} 
\usepackage{enumerate}
\usepackage[margin=1in, marginparwidth=.8in]{geometry}
\usepackage{setspace}

\usepackage[pdfauthor={}, pdftitle={}, pdftex]{hyperref}
\hypersetup{colorlinks=true}

\usepackage{fancyhdr}

\usepackage{url}

\usepackage[]{appendix}

\makeatletter
\DeclareRobustCommand\widecheck[1]{{\mathpalette\@widecheck{#1}}}
\def\@widecheck#1#2{%
    \setbox\z@\hbox{\m@th$#1#2$}%
    \setbox\tw@\hbox{\m@th$#1%
       \widehat{%
          \vrule\@width\z@\@height\ht\z@
          \vrule\@height\z@\@width\wd\z@}$}%
    \dp\tw@-\ht\z@
    \@tempdima\ht\z@ \advance\@tempdima2\ht\tw@ \divide\@tempdima\thr@@
    \setbox\tw@\hbox{%
       \raise\@tempdima\hbox{\scalebox{1}[-1]{\lower\@tempdima\box
\tw@}}}%
    {\ooalign{\box\tw@ \cr \box\z@}}}
\makeatother

\definecolor{darkred}{rgb}{0.4,0,0} 
\definecolor{darkblue}{rgb}{0,0,0.4}
\definecolor{darkgreen}{rgb}{0,.4,0}

\newcommand{\D}{\partial}
\newcommand{\CC}{\mathcal{C}} 
\newcommand{\E}{\mathbb{E}}
\newcommand{\PP}{\mathbb{P}} 
\newcommand{\EE}{\mathbb{E}} 
\newcommand{\dd}{\,\mathrm{d}} 
\newcommand{\RR}{\mathbb{R}} 
\newcommand{\NN}{\mathbb{N}} 
 
\newcommand{\fra}{\mathfrak{a}}
\newcommand{\supp}{\mathrm{supp}\,} 
\newcommand{\sgn}{\mathrm{sgn}}
\newcommand{\bR}{\mathbb{R}}
\newcommand{\tu}{\tilde{u}}

\newcommand{\al}{\alpha}
\newcommand{\eps}{\varepsilon} 

\newcommand{\VVert}[1]{{\left\vert\kern-0.25ex \left\vert\kern-0.25ex \left \vert #1 
    \right\vert\kern-0.25ex \right\vert\kern-0.25ex \right\vert}} 

\newcommand\bP{\mathbb{P}}

\newcommand\cE{\mathcal{E}}

\newcommand{\wto}{\rightharpoonup}

\newcommand{\tr}{\tilde{r}}

\newcommand*{\beq}{\begin{equation}}
\newcommand*{\eeq}{\end{equation}}

\newcommand{\bit}{\begin{itemize}}
\newcommand{\eit}{\end{itemize}}

\newcommand{\tA}{{\tilde{A}}}

\newcommand{\tfra}{{\tilde{\mathfrak{a}}}}

\renewcommand{\le}{\lesssim}
\newcommand{\tsigma}{{\tilde\sigma}}

\newcommand{\txi}{{\tilde\xi}}

\DeclareMathOperator*{\supess}{\mathrm{sup}\,\mathrm{ess}}

\theoremstyle{plain}
\newtheorem{theorem}{Theorem}[section]
\newtheorem{proposition}[theorem]{Proposition}
\newtheorem{corollary}[theorem]{Corollary}
\newtheorem{lemma}[theorem]{Lemma}

\theoremstyle{definition}
\newtheorem{definition}[theorem]{Definition}
\newtheorem{assumption}[theorem]{Assumption}

\theoremstyle{remark}
\newtheorem{remark}[theorem]{Remark}

\numberwithin{equation}{section}

\title{Ergodicity for Stochastic Porous Media Equations}

\author{Konstantinos Dareiotis, Benjamin Gess, Pavlos Tsatsoulis}

\date{}

\fancypagestyle{plain}{

  \fancyhf{}
  \fancyfoot[L]{\footnotesize \today}
  \fancyfoot[C]{\thepage}
}

\begin{document}

\maketitle

\begin{abstract}
The long time behaviour of solutions to stochastic porous media equations on smooth bounded domains with Dirichlet boundary data is studied. Based on
weighted $L^{1}$-estimates the existence and uniqueness of invariant measures with optimal bounds on the rate of mixing are proved. Along the
way the existence and uniqueness of entropy solutions is shown.


\end{abstract}

\begin{spacing}{0.05}
 \tableofcontents
\end{spacing}

\section{Introduction}

In this work we prove the existence and uniqueness of invariant measures, with optimal estimates on the rate of mixing, for stochastic porous media 
equations 
\begin{equs}
\begin{cases}
 & \partial_{t}u(t,x)=\Delta\left(|u|^{m-1}u\right)(t,x)+\sum_{k=1}^\infty\sigma^{k}(x,u(t,x))\,\dot{\beta}^{k}(t)\\
 & u(0)=\xi\\
 & u|_{\partial Q}=0
\end{cases}\label{eq:spm_intro}
\end{equs}
on smooth bounded domains $Q\subseteq\RR^{d}$, where $(\beta^{k})_{k\geq1}$ is a sequence of independent Brownian motions, $m\in (1,\infty)$, $\xi$ is the initial condition (which lies in a suitably weighted $L^1_x$ space) and $(\sigma^{k})_{k\geq1}$ 
is a sequence of H\"older continuous coefficients (for the exact assumptions see Section \ref{s:form} and Section \ref{s:ergod}). Our main result is the following contraction estimate (see Theorem \ref{thm:pol_contr} below): There 
exists a (uniform in the initial condition) constant $C>0$ such that for each two entropy solutions $u(\cdot;\xi)$, $u(\cdot;\txi)$ to \eqref{eq:spm_intro} with initial conditions $\xi$, $\tilde \xi$ respectively we have
\begin{equs}
\EE\|u(t;\xi) - u(t;\txi)\|_{L_{w;x}^{1}}\leq Ct^{-\frac{1}{m-1}}, \, t>0,
\end{equs}
were $L_{w;x}^{1}$ is the weighted $L^{1}_x$ space with weight $w$ given by the solution to $\Delta w=-1$ with zero Dirichlet boundary conditions on $Q$. This contraction estimate implies the existence and uniqueness
of an invariant measure $\mu$ to \eqref{eq:spm_intro} and the following optimal bound on the rate of mixing,
\begin{equs}
\sup_{\xi\in L^1_{w;x}}\sup_{\|F\|_{\mathrm{Lip}(L_{w;x}^{1})}\leq1}|P_{t}F(\xi)-\int_{L^1_{w;x}} F(\txi)\mu(\mathrm{d}\txi)|\leq Ct^{-\frac{1}{m-1}}, \, t>0,
\end{equs}
where $\mathrm{Lip}\left(L_{w;x}^{1}\right)$ denotes the space of Lipschitz continuous functions from $L_{w;x}^{1}$ to $\RR$ and $P_{t}$ the Markov semigroup on $B_{b}(L_{w;x}^{1})$ associated to \eqref{eq:spm_intro}
(see Theorem \ref{thm:pol_mix} below). On the route to these results we further prove the existence, uniqueness and stability of entropy solutions to \eqref{eq:spm_intro}, and the time continuity of
entropy solutions with values in $L_{\omega}^{1}L_{w;x}^{1}$ (see Theorems \ref{thm:uniq_exist}, \ref{thm:stability} and \ref{thm:L^1_cont} below).

Stochastic porous media equations of the type \eqref{eq:spm_intro} informally appear as continuum limits of interacting branching particle processes. More precisely, Méléard and Roelly have shown in \cite{MR93} that, under
appropriate rescaling, the mean field limit of branching particle processes that interact through a potential $V$ solves a non-local, non-linear stochastic diffusion equation of the type
\begin{equs}
\partial_t u(t,x)=\frac{1}{2} \Delta\left(|u|(V*u)\right)(t,x) + \sqrt{b(x,u(t,x))u(t,x)} \, \dot{W}(t), \label{eq:branch}
\end{equs}
where $\dot{W}(t)$ is space-time white noise. Informally, localising the particle interactions, by taking $V$ to a Dirac mass, leads to a stochastic partial differential equation (SPDE) of the type \eqref{eq:spm_intro}, 
albeit with space-time white noise. So far, this last step has rigorously been justified only in the deterministic case
by Lions, Mas Gallic \cite{LMG01} and Carrillo, Craig, Patacchini \cite{CCP19}.

On a broader scope, the aim of the present article is to understand the applicability of the dissipativity approach (see \cite[Section 11.6]{DPZ14} and the references therein) to the ergodicity of SPDEs with multiplicative noise. While the focus is on stochastic porous medium equations 
the ideas are equally relevant for the case of semi-linear SPDEs (see Remark \ref{rmk:semilinear} below). Let us briefly and informally recall the dissipativity approach in the case of additive noise, that is, let $u$, $\tu$
be solutions to 
\begin{equs}
\partial_t u(t,x)= \Delta\left(|u|^{m-1}u\right)(t,x)+ G\dot{W}(t),
\end{equs}
with zero Dirichlet boundary conditions on $Q$ and diffusion coefficients $G$. Then, informally, using Lemma \ref{lem:lwr_bd} below we have 
\begin{equs}
 \partial_t \|u-\tu\|_{H^{-1}_x}^{2} & =2 \left(\Delta\left(|u|^{m-1}u\right)-\Delta\left(|\tu|^{m-1}\tu\right),u-\tu\right)_{H^{-1}_x}
 = - 2 (|u|^{m-1}u-|\tu|^{m-1}\tu,u-\tu)_{L^2_x} \\
 & \leq -2 C_{m} \|u-\tu\|_{L^{m+1}_x}^{m+1} 
 \leq -2 C_{m} \|u-\tu\|_{H^{-1}_x}^{m+1}, \label{eq:intro-contraction}
\end{equs}
where $H^{-1}_x:=(H_{0;x}^{1})^{*}$ and $C_m>0$. This implies contraction estimates of the type
\begin{equs}
\|u(t)-\tu(t)\|_{H^{-1}_x}^{2} & \leq C_{m} t^{-\frac{1}{m-1}}, \, t>0,
\end{equs}
and polynomial rates of mixing
\begin{equs}
\sup_{\xi\in H^{-1}_x} \sup_{\|F\|_{\mathrm{Lip}(H^{-1}_x)}\leq 1} \left|P_t F(\xi) - \int_{H^{-1}_x} F(\txi) \mu(\mathrm{d}\txi)\right|\leq Ct^{-\frac{1}{m-1}}, \, t>0.
\end{equs}
This argument is restricted to additive noise. Indeed, if we consider \eqref{eq:spm_intro} then, following the previous computations,
\begin{equs}
\partial_t \E\|u-\tu\|_{H^{-1}_x}^{2} & \leq - 2 C_{m} \E\|u- \tu\|_{H^{-1}_x}^{m+1} + 
\sum_{k=1}^\infty \E\|\sigma^{k}(u)-\sigma^{k}(\tu)\|_{H^{-1}_x}^{2}.
\end{equs}
Even when $u\mapsto\sigma^{k}(\cdot,u)$ is Lipschitz continuous in $H^{-1}_x$ (which is rarely the case, cf.~the
discussion in \cite{DGG19}), it is unclear how to prove stability without a smallness assumption on the Lipschitz 
constant of the coefficients $\sigma^{k}$.

The main insight of the present work is that the weighted topology $L_{w;x}^{1}$ introduced here is better adapted to the dissipativity approach for SPDEs with multiplicative noise.
In fact, it is shown that the stochastically perturbed equation enjoys the same stability properties as the deterministic PDE when considered in this weighted topology.

\begin{remark}\label{rmk:semilinear} The same questions can be asked in the case of semi-linear SPDEs, such as, 
\begin{equs}
\partial_t u(t,x)=(\Delta u(t,x)+f(u(t,x))) + \sum_{k=1}^\infty\sigma^{k}(x,u(t,x))\,\dot{\beta}^{k}(t), \label{eq:semilinear}
\end{equs}
with zero Dirichlet boundary conditions. For simplicity let us assume that $f:\RR\to\RR$ is non-decreasing. Again, an $L^{2}$-based approach suffers
from the It\^o-correction terms, since, informally,
\begin{equs}
 &\partial_t \E\|u-\tu\|_{L^2_x}^{2} 
 \\ 
 & \quad = 2 \E(\Delta(u-\tu) + f(u)-f(\tu), u-\tu)_{L^2_x} + \sum_{k=1}^\infty \E\|\sigma^{k}(u)-\sigma^{k}(\tu)\|_{L^2_x}^{2}\\
 & \quad =-2\E\|\nabla(u-\tu)\|_{L^2_x}^{2}+ 2 \E(f(u)-f(\tu),u-\tu)_{L^2_x} + \sum_{k=1}^\infty \|\sigma^{k}(u)-\sigma^{k}(\tu)\|_{L^2_x}^{2}\\
 & \quad \leq - 2C_P \E\|u-\tu\|_{L^2_x}^{2} + \sum_{k=1}^\infty\|\sigma^{k}(u)-\sigma^{k}(\tu)\|_{L^2_x}^{2},
\end{equs}
where $C_P$ is the Poincar\'e constant. This implies stability only if the Lipschitz constant of $u\mapsto\sigma^{k}(\cdot,u)$ is small enough. Hence, even if the deterministic PDE is stable,
this is not necessarily inherited by the stochastically perturbed equation.

However, the weighted norm $\|\cdot\|_{L_{w;x}^{1}}$ is better adapted to study the stability of \eqref{eq:semilinear}, in the sense that one expects the estimate 
\begin{equs}
 \partial_t \E\|u-\tu\|_{L_{w;x}^{1}} = -\E\|u-\tu\|_{L_{x}^{1}} + \E \int_Q (f(u)-f(\tu)) \sgn(u-\tu) w \dd x
 \leq - C \E\|u-\tu\|_{L_{w;x}^{1}},
\end{equs}
which immediately implies exponential mixing, with the same rate as in the deterministic case. 
\end{remark}

\subsection{Comments on the literature}

The existence, uniqueness and mixing properties of solutions to stochastic porous medium equations have attracted considerable attention in the literature. However, all available results
are essentially restricted to additive noise in the following sense: Either, only purely additive noise can be treated or it is assumed that the noise contains a sufficiently
non-degenerate additive part. Therefore, SPDEs of the type \eqref{eq:spm_intro} were out of reach of existing results.

The available results can be categorised in two classes: The first class of results relies on the dissipativity approach exploiting the contractive properties of the deterministic
porous medium equation, while the second class of results relies on the mixing effects of the random perturbation. As such, these two classes lead to essentially different assumptions
and results.

The first results on the existence of invariant measures for stochastic porous medium equations were obtained by Da Prato, Röckner \cite{DPR04-2,DPR04} and Bogachev, Da Prato, Röckner
\cite{BDPR04}. The dissipativity approach to prove the existence, uniqueness and rates of mixing was first applied to stochastic porous medium equations by Da Prato, Röckner, Rozovskii,
Wang in \cite{DPRRW06}, see also the more recent monograph by Barbu, Da Prato, Röckner \cite{BDPR16}. As explained above, this approach is restricted to additive noise. Since the obtained
estimates are based on the contractive properties of the (deterministic) porous medium operator, no non-degeneracy assumptions on the noise have to be assumed and the obtained rates of mixing
are of polynomial type. In \cite{BGLR11} the dissipativity approach was further used by Beyn, Gess, Lescot, Röckner in order to prove that the random attractor consists of a single random point.
A generalisation of the dissipativity approach, based on coupling arguments, has been introduced by Gess, Tölle in \cite{GT14,GT16} allowing to prove the ergodicity of generalised porous medium 
equations in cases where no strict contraction estimates, such as \eqref{eq:intro-contraction}, apply.

Concerning the second class of results, in the case of purely additive noise
\begin{equs}
\partial_t u(t,x) = \Delta\left(|u|^{m-1}u\right)(t,x) + G\dot{W}(t)
\end{equs}
couplings by change of measure were constructed by Wang in \cite{Wan07,Wan13}, see also Liu \cite{Li09}. This construction relies on a non-degeneracy assumption on the diffusion coefficients $G$. In particular,
it has to be assumed that $G$ is surjective onto the energy space $L^{m+1}_x$. Since this assumption competes with the smoothness assumption on the noise required by the well-posedness theory (cf.~e.g.~Liu,
Röckner \cite{LR15}), this restricts the applicability of these results to one spatial dimension. Because this approach exploits the non-degeneracy of the noise by means of establishing Harnack inequalities 
on the resulting Markov semigroup, the associated rate of mixing is of exponential type, that is, there is a unique invariant measure $\mu$ and constants $\lambda>0$, $C\geq 0$ such that
\begin{equs}
\sup_{\xi}\|P_{t}\delta_\xi-\mu\|_{TV}\leq C e^{-\lambda t}, \, t>0.
\end{equs}
Lower bounds on the exponential rate $\lambda$ have been obtained by Wang in \cite{Wan15b}. Flandoli, Gess, and Scheutzow in \cite{FGS17} and Gess in \cite{Ge13} used these methods in order to prove synchronisation by
noise, in the sense that the random attractor was shown to consist of a single random point. This line of arguments has been further improved by Wang in \cite{Wan15a}, where coupling by change of measure has been replaced 
by reflection coupling, which allowed to also include perturbations by multiplicative noise, while retaining the non-degeneracy assumption on the additive noise part, that is,
\begin{equs}
\partial_t u(t,x) = \Delta\left(|u|^{m-1}u\right)(t,x) + B(u)\dot{W}^{1}(t) + G \dot{W}^{2}(t),
\end{equs}
with $G$ being non-degenerate as above. This work seems to be the only previous result on ergodicity for stochastic porous medium equations with multiplicative noise in the literature.
However, in this work $B$ is still assumed to be Lipschitz continuous in $H^{-1}_x$, which in the case of $B(u) = \sigma(x,u(x))$ (as in \eqref{eq:spm_intro}) essentially implies that $\sigma$ is linear in $u$. 

We conclude the discussion of available results on the ergodicity of stochastic porous medium equations by emphasising that there are no previous results on the ergodicity of \eqref{eq:spm_intro}
with purely non-linear multiplicative noise, as it appears in \eqref{eq:branch} .


The well-posedness of entropy solutions to \eqref{eq:spm_intro} on the torus has been recently shown by Dareiotis, Gerencsér, Gess in \cite{DGG19}. Already in the context of the well-posedness of solutions it has been
realised in \cite{DGG19} that $L_{x}^{1}$ appears to be better suited than $H^{-1}_x$, which was used in many previous works, see for example \cite{LR15} and the references therein, since the non-linear diffusion coefficients
behave nicely in $L_{x}^{1}$ while they are not expected to be even Lipschitz continuous in $H^{-1}_x$. The proof of well-posedness of entropy solutions to \eqref{eq:spm_intro} on bounded domains, given in the present work, 
builds upon the analysis from \cite{DGG19}. In contrast to the periodic case considered in \cite{DGG19}, the presence of the boundary introduces the need of a weighted $L^{1}$-norm in order to control boundary terms. 

\subsection{Organisation of the article} 

In Section \ref{s:form} we set up the right formulation to study the well-posedness of \eqref{eq:spm_intro}. In Section \ref{s:main_results} we present our main results. 

In Section \ref{s:well-posedness} we discuss the well-posedness of stochastic porous media equations on bounded domains. We begin with the existence, uniqueness and $L^1$-contraction of 
entropy solutions (see Theorem \ref{thm:uniq_exist}), stability with respect to the data (see Theorem \ref{thm:stability}) and continuity in $L^1_\omega L^1_{w;x}$ 
(see Theorem \ref{thm:L^1_cont}). These results hold for more general porous medium operators $\Delta A$ (see Assumption \ref{as:A}-\ref{as:A first}) and the
initial condition $\xi$ is assumed to be in $L^{m+1}_\omega L^{m+1}_x$. As a corollary of Theorems \ref{thm:uniq_exist} and \ref{thm:L^1_cont}  we present an extension result to initial conditions 
in $L^1_\omega L^1_{w;x}$ (see Proposition \ref{prop:def-of-v}). In Section \ref{s:ergod} we present a bound for the solutions in $L^{m+1}_\omega L^{m+1}_x$ which is
uniform in the initial condition and in time (see Proposition \ref{prop:L^m+1_bddness}). We then discuss the main contraction estimate (see Theorem \ref{thm:pol_contr}), the Markov property 
(see Proposition \ref{prop:Markov}) and mixing to a unique equilibrium with optimal rate (see Theorem \ref{thm:pol_mix}).

In Section \ref{s:basic_est} we prove some important $L^1$-estimates, in Section \ref{s:well-posedness_proof} we prove the results listed in Section \ref{s:well-posedness}
and in Section \ref{s:ergod_proof} we prove the results listed in Section \ref{s:ergod}. Many technical results and proofs can be found in the \hyperlink{appendix.0}{Appendix}.

\subsection{Notation}

For a set $S\subset \RR^d$ and $m\in \NN$, we denote by $C^m(S)$ the space of $m$-times differentiable functions on $S$ and by $C^m_c(S)$ the 
subset of $m$-times differentiable compactly supported functions on $S$. Given a variable $\mathfrak{s}\in S$, $p\in[1,\infty]$ and 
$m\in\RR$, we denote by $L^p_\mathfrak{s}$ and $W^{m,p}_\mathfrak{s}$ the usual $L^p$ and $W^{m,p}$ spaces of functions in
this variable. If $p=2$ we write $H^m_\mathfrak{s}$ instead of $W^{m,2}_\mathfrak{s}$. If in addition $m\in \NN$, we write $H^m_{0;\mathfrak{s}}$
for the closure of $C^m_c(S)$ in $H^m_\mathfrak{s}$. For a probability space $(\Omega, \mathcal{F}, \PP)$ and $p\in[1,\infty)$
we denote by $L^p_\omega$ the space of $p$-integrable random variables in $\omega \in\Omega$. If the probability space carries a filter $(\mathcal{F}_t)_{t\geq 0}$,
we denote by $L^p_{\omega,t}$ the space of predictable $p$-integrable random processes in $(\omega,t)\in \Omega \times [0,T]$, for every $T>0$. For spaces of functions
of several arguments we sometimes use mixed notation. For example, $L^p_\omega L^q_x$ stands for the space of $p$-integrable random variables taking values in $L^q_x$,
while $L^p_{\omega,t}L^q_x$ stands for the space of predictable $p$-integrable random processes taking values in $L^q_x$.    

We denote by $w$ the solution to the boundary value problem
\begin{align} \label{eq:w}
 \begin{cases}
  & \Delta w = -1 \\
  & w|_{\partial Q} = 0 
 \end{cases}
 .
\end{align}
It is well-known that $w\in H^1_{0;x}$ and $w>0$ in $Q$. We define $L^1_{w;x}$ to be the space of measurable functions $f:Q\to \RR$ such that 
\begin{equs}
 \|f\|_{L^1_{w;x}} := \int_Q |f(x)| w(x) \dd x <\infty. \label{eq:L^1_weighted}
\end{equs}

We define the set 
\begin{equs}
 \mathcal{S} := \{d,K,m,|Q|,T\} \label{eq:S}
\end{equs}
where $d$ denotes the dimension of the $x$-space, $K$ and $m$ are positive constants given by Assumptions \ref{as:A} and \ref{as:noise} 
below, $|Q|$ denotes the volume of $Q$ and $T>0$.

Throughout the article, $C$ denotes a strictly positive constant which depends on the structural set $\mathcal{S}$, unless 
otherwise stated, and might change from line to line. In the proofs we will frequently use the notation $a \le b$ by which 
we mean $a \leq C b $. The notation $a\le_q b$ (respectively $C_q$) means that the constant $C$ depends on $\mathcal{S}$ and
on $q$. We also write $a\vee b$ (respectively $a\wedge b$) to denote the maximum (respectively minimum) between $a$ and $b$.  

In the proofs we sometimes use the abbreviation $\int_x$, $\int_t$ instead of $\int_Q \dd x$, $\int_0^T \dd t$. An analogous notation is
used for multiple integrals. For example, $\int_{t,x}$ stands for $\int_0^T \int_Q \dd x \dd t$. 

\subsection*{Acknowledgements}

Financial support by the DFG through the CRC 1283 ``Taming uncertainty and profiting from randomness and low regularity in analysis, stochastics and their applications''
is acknowledged.

\section{Formulation} \label{s:form}

We consider a generalisation of \eqref{eq:spm_intro} in the form 
\begin{equation}\label{eq:spm}
 \begin{cases}
 & \partial_t u(t,x) = \Delta A(u(t,x)) + \sigma^k(x,u(t,x)) \, \dot{\beta}^k(t) \\
 & u(0) = \xi \\
 & u|_{\partial Q} = 0
 \end{cases}
\end{equation}
under Assumptions \ref{as:A} and \ref{as:noise} below on a smooth bounded domain $Q\subset \RR^d$. From now on we fix a filtered probability space $(\Omega,\mathcal{F},\{\mathcal{F}_t\}_{t\geq 0}, \PP)$
with a sequence $(\beta^k)_{k\geq 1}$ of independent Brownian motions.

For a locally integrable function $f:\RR\to \RR$ we define 
\begin{equation*}
 [f](r) : = \int_0^r f(\zeta) \dd \zeta, \, [f,c](r) : = \int_c^r f(\zeta) \dd \zeta.
\end{equation*}
We also let $\fra : = \sqrt{A'}$.

\begin{assumption}\label{as:A}
The following hold for some $K\geq 1$ and $m>1$. 
\begin{enumerate}[a.]
\item \label{as:A first}
The function $A:\bR\mapsto\bR$ is differentiable, strictly increasing and odd. The function $\fra$ is differentiable away from the origin, and satisfies the bounds
\begin{equ}\label{eq:as fra}
|\fra(0)|\leq K,
\, |\fra'(r)|\leq K|r|^{\frac{m-3}{2}} \text{ if } r>0,
\end{equ}
as well as
\begin{equation}\label{eq:as Psi}
K\fra(r)\geq I_{|r|\geq 1},\,
  K|[\fra](r)-[\fra](\tr)|\geq\left\{
  \begin{array}{@{}ll@{}}
    |r-\tr|, & \text{if } |r|\vee|\tr|\geq 1 \\
    |r-\tr|^{\frac{m+1}{2}}, & \text{if } |r|\vee|\tr|< 1
  \end{array}\right..
\end{equation} 

\item \label{as:ic}
The initial condition $\xi$ is an $\mathcal{F}_0$-measurable $L^{m+1}_x$-valued random variable such that $\E \| \xi \|_{L^{m+1}_x}^{m+1}< \infty$.
\end{enumerate}
\end{assumption}

\begin{assumption} \label{as:noise}
The function $\sigma:Q\times\bR\mapsto\ell^2$ satisfies the following bounds. There exist $\kappa\in (0,\frac{1}{2}]$, $\bar \kappa\in (\frac{1}{m\wedge 2},1]$ and $K\geq1$ such that for every $r\in\bR$, $\tr\in[r-1,r+1]$ and $x,y\in Q$,
\begin{equs}
|\sigma(x,r)|_{\ell^2}\leq K(1+|r|),\,
|\sigma(x,r)-\sigma(y,\tr)|_{\ell^2}\leq K|r-\tr|^{\frac{1}{2}+\kappa}+ K (1+|r|) |x-y|^{\bar\kappa}.
\end{equs}
\end{assumption}
For $\sigma, \tsigma:Q\times\bR\mapsto\ell^2$ satisfying Assumption \ref{as:noise} we set 
\begin{equs}
d(\sigma,\tsigma):= \sup_{x\in Q, \, r\in \RR} \frac{|\sigma(x,r)-\tsigma(x,r)|_{\ell^2}^2}{(1+|r|)^{m+1}}.
\end{equs}

\begin{definition}\label{def:adm_fct} A pair of functions $(\eta,\phi)$ is called admissible if
\begin{enumerate}[i.]
 \item $\eta\in C^2(\RR)$, $\eta''\geq 0$ and $\supp \eta''$ is compact.
 \item $\phi\geq 0$, $\phi = \varphi \varrho$ where $\varphi\in C^\infty_c([0,T))$ and $\varrho\in C^\infty_c(Q)$ .
\end{enumerate}
\end{definition}

Similarly to \cite{DGG19} we give the following definition of entropy solutions for \eqref{eq:spm}.

\begin{definition} \label{def:entr_sol} A predictable stochastic process $u:\Omega\times [0,\infty) \to L^{m+1}_x$ is an entropy solution to \eqref{eq:spm} if
\begin{enumerate}[i.]
 \item \label{item:solution-in-spaces} $u\in L^{m+1}_{\omega,t}L^{m+1}_x$ and $A(u) \in L^2_{\omega,t} H^1_{0;x}$. 
 \item \label{item:solution-in-spaces-2} For every $f\in C(\RR)$ bounded we have $[\fra f](u)\in L^2_{\omega,t} H^1_x$ and $\partial_{x_i}[\fra f](u) = f(u) \partial_{x_i}[\fra](u)$.  
 \item \label{item:entropies} For every admissible pair of functions $(\eta,\phi)$ as in Definition \ref{def:adm_fct} we have that
 \begin{align*}
  - \int_0^T \int_Q \eta(u) \partial_t \phi \dd x \dd t & \leq \int_Q \eta(\xi) \phi(0) \dd x + \int_0^T \int_Q \left[\eta'\fra^2\right](u) \Delta \phi \dd x \dd t \\
  & \quad + \int_0^T \int_Q \left(\frac{1}{2} \phi \eta''(u) |\sigma(u)|_{\ell^2}^2 - \phi \eta''(u) |\nabla [\fra](u)|^2\right) \dd x \dd t \\
  & \quad + \int_0^T \int_Q \phi \eta'(u) \sigma^k(u) \dd x \dd \beta^k(t).
 \end{align*} 
\end{enumerate}
\end{definition}

We refer to \eqref{eq:spm} as $\mathcal{E}(A,\sigma,\xi)$. In the sequel we write $u(\cdot;\xi)$ to denote the solution of $\mathcal{E}(A,\sigma,\xi)$. If the value of 
the initial condition is clear from the context, we simply write $u$.  

\section{Main results} \label{s:main_results}

\subsection{Well-posedness} \label{s:well-posedness}

The next two theorems build upon the analysis from \cite{DGG19} and concern the existence, uniqueness and stability of entropy solutions to 
\eqref{eq:spm}. All the results in this subsection hold for fixed, but arbitrarily large, $T>0$.  

\begin{theorem} \label{thm:uniq_exist} Let Assumptions \ref{as:A} and \ref{as:noise} hold. Then, there exists a unique entropy solution 
$u$ of $\mathcal{E}(A,\sigma, \xi)$. Moreover, if $u(t;\txi)$ is an entropy solution of $\mathcal{E}(A,\sigma, \txi)$ the following contraction
estimate holds, 
\begin{equation}
 \supess_{t\in[0,T]} \EE \|u(t;\xi) - u(t;\txi)\|_{L^1_{w;x}} \leq \EE \|\xi - \txi\|_{L^1_{w;x}}. \label{eq:main_contraction}
\end{equation}
\end{theorem}

\begin{theorem} \label{thm:stability}
Let $(A_n)_{n\geq 1}$ and $(\xi_n)_{n\geq 1}$ satisfy Assumption \ref{as:A} and let $(\sigma_n)_{n\geq 1}$ satisfy Assumption \ref{as:noise}, uniformly in $n$.
Assume furthermore that for $n \to \infty$, $A_n\to A$ uniformly on compact sets, $\xi_n\to\xi$ in $L^{m+1}_\omega L^{m+1}_x$, and $d(\sigma_n,\sigma) \to 0$, for some $A$
and $\xi$ satisfying Assumption \ref{as:A} and $\sigma$ satisfying Assumption \ref{as:noise}. Let $u_n, u$ be the  unique entropy 
solutions of $\cE(A_n,\sigma_n,\xi_n)$, $\cE(A,\sigma,\xi)$. Then, as $n \to \infty$, $u_n\to u$ in $L^1_{\omega,t}L^1_{w;x}$.
\end{theorem}

In the next theorem we prove the continuity of entropy solutions in $L^1_\omega L^1_{w;x}$.

\begin{theorem} \label{thm:L^1_cont} Let Assumptions \ref{as:A} and \ref{as:noise} hold. If $u$ is an entropy solution of $\mathcal{E}(A,\sigma,\xi)$, then $u\in C([0,T];L^1_\omega L^1_{w;x})$.
\end{theorem}

As a corollary of the contraction estimate \eqref{eq:main_contraction} and the continuity in $L^1_\omega L^1_{w;x}$ we have the following extension result.  

\begin{proposition} \label{prop:def-of-v} Let Assumptions \ref{as:A} and \ref{as:noise} hold. The mapping 
\begin{equs}
 L^{m+1}_\omega L^{m+1}_x\ni \xi \mapsto u(\cdot;\xi)\in  C([0,T];L^1_\omega L^1_{w;x})
\end{equs}
extends uniquely to a continuous map from $L^1_\omega L^1_{w;x}$ to $C([0,T];L^1_\omega L^1_{w;x})$. Furthermore, the following contraction estimate 
holds for every $\xi, \txi\in L^1_\omega L^1_{w;x}$,
\begin{equs} 
 \sup_{t\in [0,T]} \EE\|u(t;\xi) - u(t;\txi)\|_{L^1_{w;x}} & \leq \EE\|\xi - \txi\|_{L^1_{w;x}}. \label{eq:v_contr}
\end{equs}
\end{proposition}

\subsection{Ergodicity} \label{s:ergod}

In this section we assume that $A(r)= |r|^{m-1}r$, for $m\in(1,\infty)$. From now on, it is more convenient to work with initial conditions $\xi\in L^1_{w;x}$. In the following, we define $u(\cdot;\xi)$
for $\xi\in L^1_{w;x}$ by continuity using Proposition \ref{prop:def-of-v}. By Remark \ref{rem:L^1_ext} and Proposition \ref{prop:L^m+1_bddness} below, one can prove that for every $\xi\in L^1_\omega L^1_{w;x}$ the extension 
$u(\cdot;\xi)$ is an entropy solution of \eqref{eq:spm} on $(0,T]$ for every $T>0$, in the sense that it satisfies Definition \ref{def:entr_sol} with $[0,T]$ replaced by $[s,T]$, for 
every $s\in(0,T]$.  

\begin{remark} \label{rem:L^1_ext} Although Proposition \ref{prop:def-of-v} allows to extend $u(\cdot;\xi)$ for $\xi\in L^1_\omega L^1_{w;x}$ by continuity, it is unclear 
whether the extension solves \eqref{eq:spm} in general. However, it is easy to see that if there exists a sequence $\xi_n\to\xi$ such that for every $s>0$
\begin{equs}
 \sup_{n\geq 1}\sup_{t\in[s,T]} \EE\|u(t;\xi_n)\|_{L^{m+1}_x} <\infty, \label{eq:L^m+1_bound_xi}
\end{equs}
then $u(\cdot;\xi)$ is an entropy solution of \eqref{eq:spm} on $(0,T]$. Indeed, since $u(\cdot;\xi_n) \to u(\cdot;\xi)$ in $C([0,T];L^1_\omega L^1_{w;x})$, 
we know that for every $s\in(0,T]$, passing to a subsequence, $u(s,x;\xi_n) \to u(s,x;\xi)$ for almost every $(\omega,x)$. Hence, by Fatou's lemma we have 
\begin{equs}
 \EE\|u(s;\xi)\|_{L^{m+1}_x} \leq \liminf_{n\to \infty} \EE\|u(s;\xi_n)\|_{L^{m+1}_x} \leq \sup_{n\geq 1} \EE\|u(s;\xi_n)\|_{L^{m+1}_x}
\end{equs}
and the latter quantity is uniformly bounded in $n$ due to \eqref{eq:L^m+1_bound_xi}. Hence, by Theorem \ref{thm:uniq_exist} there exists a unique solution
of \eqref{eq:spm} on $[s,T]$ with initial condition $u(s;\xi)$, which we denote by $u_s(\cdot;u(s;\xi))$. By Corollary \ref{cor:restart} we know that $u(\cdot;\xi_n)$
coincides with $u_s(\cdot;u(s;\xi_n))$ (the unique entropy solution of \eqref{eq:spm} with initial condition $u(s;\xi_n)$) on $[s,T]$. Since $u(s;\xi_n) \to u(s;\xi)$ in $L^1_\omega L^1_{w;x}$,
using \eqref{eq:v_contr} we see that $u_s(\cdot;u(s;\xi_n)) \to u_s(\cdot;u(s;\xi))$ in $C([s,T];L^1_\omega L^1_{w;x})$. But $u(t;\xi_n) = u_s(t;u(s;\xi_n))$ for $t\in[s,T]$, 
which in turn implies that $u(t;\xi) = u_s(t;u(s;\xi))$. Since $s\in(0,T]$ is arbitrary this proves that $u(\cdot;\xi)$ is an entropy solution on $(0,T]$.  
\end{remark}

The next proposition states that entropy solutions satisfy the so-called ``coming down from infinity'' property, which implies \eqref{eq:L^m+1_bound_xi}.

\begin{proposition} \label{prop:L^m+1_bddness} Let Assumption \ref{as:noise} hold. Then, 
\begin{equs}
 \sup_{\xi\in L^{m+1}_{\omega, x}}\sup_{t\in[0,\infty)} (t\wedge1)^{\frac{m+1}{m-1}} \EE\|u(t;\xi)\|_{L^{m+1}_x}^{m+1} < \infty. \label{eq:coming_down_infty}
\end{equs}
\end{proposition}

\begin{remark} \label{rem:coming_down_infty_L^1} By a simple application of Fatou's lemma we can replace the supremum over $\xi\in L^{m+1}_{\omega, x}$ by a supremum over $\xi\in L^1_\omega L^1_{w;x}$
in \eqref{eq:coming_down_infty}.
\end{remark}

The choice $A(r) = |r|^{m-1}r$ allows us to obtain a polynomial (with respect to time) decay for differences of entropy solutions
uniformly in the initial conditions, as shown in the next theorem. 

\begin{theorem} \label{thm:pol_contr} Let Assumption \ref{as:noise} hold and $m_*=\frac{m}{m-1}$. There exists $C>0$ depending only on $m$ such that for all $t \geq 0$ we have 
  \begin{equs}
 \sup_{ \xi, \tilde{\xi} \in L^1_\omega L^1_{w;x} }\EE\|u(t; \xi) - u(t, \txi)\|_{L^1_{w;x}} \leq C \|w\|_{L^{m_*}_x}^{m_*} t^{-\frac{1}{m-1}}.
  \end{equs}
\end{theorem}

\begin{remark} \label{rem:opt} The decay rate in Theorem \ref{thm:pol_contr} is optimal. Indeed, one can consider the homogeneous porous medium
equation $\partial_t u = \Delta \left( |u|^{m-1} u\right)$ with Dirichlet boundary conditions and search for solutions of the form  
$u(t,x) = (1+t)^{-\frac{1}{m-1}} f(x)$. It is easy to check that $u$ is a solution if $f$ satisfies the following equation,
\begin{equs}
 \begin{cases}
  & \Delta \left(|f|^{m-1} f\right) + \frac{1}{m-1} f = 0 \\
  & f|_{\partial Q} = 0.
 \end{cases}
 \label{eq:f}
\end{equs}
Existence and regularity of non-zero solutions to \eqref{eq:f} for sufficiently smooth domains was studied in \cite{AP81}. In 
particular, \cite[Proposition 1]{AP81} implies the existence of a non-zero solution $f$ which is bounded on $\partial \Omega \cup \Omega$.
\end{remark}

Below we let $B_b(L^1_{w;x})$ be the space of bounded Borel measurable functions from $L^1_{w;x}$ to $\RR$. We need the following definition. 

\begin{definition} \label{def:semigroup} We define $P_t: B_b(L^1_{w;x}) \to B_b(L^1_{w;x})$ by
\begin{equs}
 P_tF(\xi):= \EE F(u(t;\xi))
\end{equs}
where $F\in B_b(L^1_{w;x})$ and $\xi\in L^1_{w;x}$ . 
\end{definition}

\begin{proposition} \label{prop:Markov}  Let Assumption \ref{as:noise} hold. The family $(P_t)_{t \geq 0} $ is a Feller Markov semigroup on $B_b(L^1_{w;x})$. 
\end{proposition}

\begin{remark} One can actually define $P_t: B_b(L^1_{w;x}) \to B_b(L^1_{w;x})$ for any $A$ satisfying Assumption \ref{as:A}-\ref{as:A first} 
using Proposition \ref{prop:def-of-v}. It is clear from the proof of Proposition \ref{prop:Markov} that $P_t$ is a Feller Markov semigroup on
$B_b(L^1_{w;x})$ even in this case.
\end{remark}

Theorem \ref{thm:pol_contr} provides a quantitative estimate for the semigroup $P_t$ acting on Lipschitz continuous functions on $L^1_{w;x}$ and allows to 
prove the existence and uniqueness of an invariant measure (which is actually supported on $L^{m+1}_x$). It also provides optimal mixing rates (see Remark \ref{rem:opt})
uniformly in the initial condition. We summarise in the following theorem. 

\begin{theorem} \label{thm:pol_mix} Let Assumption \ref{as:noise} hold and $m_*=\frac{m}{m-1}$. There exists a unique invariant measure $\mu\in \mathcal{M}_1(L^1_{w;x})$ for the 
semigroup $P_t$ which moreover is supported on $L^{m+1}_x$. Furthermore, there exist $C>0$, depending only on $m$, such that for all $t \geq 0$
  \begin{align*}
  \sup_{\xi \in L^1_{w;x}}\sup_{\|F\|_{\mathrm{Lip}(L^1_{w;x})} \leq 1} |P_t F(\xi) - \int_{L^1_{w;x}} F(\txi) \mu(\mathrm{d} \txi)| \leq 
  C \|w\|_{L^{m_*}_x}^{m_*} t^{-\frac{1}{m-1}},
  \end{align*}
where $\mathrm{Lip}\left(L_{w;x}^{1}\right)$ is the space of Lipschitz continuous functions from $L_{w;x}^{1}$ to $\RR$.
\end{theorem}

\section{The \texorpdfstring{$(\star)$}{star}-property and \texorpdfstring{$L^1$}{L1}-estimates} \label{s:basic_est}

In this section we introduce the $(\star)$-property, a purely technical concept, and derive the basic $L^1$-estimates which we use in later sections to prove our main 
results. Before we proceed we need some notation. 

Below we fix a non-negative compactly supported smooth function $\rho:\RR \to \RR$ supported in $(0,1)$ such that $\int \rho(t) \dd t = 1$ and for $\theta \in(0,1)$ we set $\rho_\theta (t) := \theta^{-1} \rho(\theta^{-1}t)$.
For $x\in \RR^d$ we also let $\varrho(x) := \prod_{i=1}^d \rho(x_i)$ and for $\eps\in (0,1)$ we set $\varrho_\eps(x) := \eps^{-d} \varrho(\eps^{-1}x)$.  

For $g\in C^\infty(\RR)$ with $\supp g'$ compact, $\tilde\varrho\in C^\infty_c(Q\times Q)$, $\varphi\in C^\infty_c((0,T))$, $\tilde \sigma$ as in Assumption 
\ref{as:noise}, predictable random variable $\tilde u\in L^{m+1}_{\omega,t}L^{m+1}_x$, $\theta>0$ and $a\in \RR$ we set
\begin{align*}
 & \phi_\theta(t,x,s,y) := \tilde\varrho(x,y) \rho_\theta(t-s) \varphi\left(\frac{t+s}{2}\right) \\
 & F_\theta(t,x,a) := \int_0^T \int_Q \tilde\sigma^k(y,\tu(s,y)) g(\tu(s,y) - a) \phi_\theta(t,x,s,y) \dd y \dd \beta^k(s),
\end{align*}
with a slight abuse of notation since we hide the dependence of $\phi_\theta$ and $F_\theta$ on the various functions.

We need the following definition.

\begin{definition} \label{def:star_prop} We say that a predictable random variable $u\in L^{m+1}_{\omega,t}L^{m+1}_x$ has the $(\star)$-property with coefficient $\sigma$
if for every $g,\varrho,\varphi,\tsigma, \tu$ as above and for every $\theta>0$ sufficiently small, we have that $F_\theta(\cdot,\cdot,u)\in L^1_\omega L^1_{t,x}$
and 
\begin{equs}
 & \EE \int_0^T \int_Q F_\theta(t,x,u(t,x)) \dd x \dd t\\
 & \quad \leq - \EE \int_{[0,T]^2} \int_{Q^2} \sigma^k(x,u(s,x)) \tsigma^k(y, \tu(s,y)) g^{\prime} (u(s,x) - \tu(s,y)) \phi_\theta(t,x,s,y) \dd x \dd y \dd t \dd s
 + C \theta^{1-\mu} \\ 
 \label{eq:star_prop}
\end{equs}
for $\mu = \frac{3m+5}{4(m+1)}$ and some constant $C>0$ (independent of $\theta$).
\end{definition}

The main result of this section is Lemma \ref{lem:final_psi} which is the counterpart of \cite[Theorem 4.1]{DGG19} for Dirichlet boundary conditions.
For the readers convenience we split the proof of Lemma \ref{lem:final_psi} into Proposition \ref{prop:basic_est} and Lemma \ref{lem:almost_final}.

From now on, for $\alpha, \delta, \eps \in (0,1)$ and $\lambda \geq 0$, we set 
\begin{equs} \label{eq:def-G}
\mathcal{G}_\alpha(\delta, \eps, \lambda) := \delta^{2\kappa} + \delta^{-1} \eps^{2\bar\kappa} 
+ \delta \eps^{-1} + \delta^{2\alpha} \eps^{-2} + \eps^{-2} \lambda^2 + \eps^{-1} \lambda.
\end{equs}

\begin{proposition} \label{prop:basic_est} Let $u, \tu$ be entropy solutions of the Dirichlet problems 
$\mathcal{E}(A,\sigma,\xi)$, $\mathcal{E}(\tA,\tsigma,\txi)$, where the data satisfy  Assumptions \ref{as:A} and \ref{as:noise},
and  assume that $u$ satisfies the $(\star)$-property with coefficient $\sigma$ (see Definition \ref{def:star_prop}). Then, for every 
non-negative functions $\psi \in C^\infty_c(Q)$, $\varphi\in C^\infty_c((0,T))$, and $\al\in(0,\frac{1}{2})$ there exists $C\equiv C(\mathcal{S}, \alpha, \psi, \varphi )>0$ 
such that for every $\lambda,\eps,\delta \in (0,1)$ we have that
\begin{equs}
 & - \EE\int_0^T \int_{Q^2} |u(t,x) - \tu(t,y)|  \partial_t\varphi(t) \psi(x)  \varrho_\eps(x-y) \dd x \dd y \dd t \\
 & \quad \leq \EE \int_0^T \int_{Q^2} |A(u(t,x)) - \tA(\tu(t,y))| \varphi(t) \Delta \psi(x) \varrho_\eps(x-y) \dd x \dd y \dd t  \\
 & \quad \quad + C \eps^{-2} \EE\left(\|\mathbf{1}_{|u|\geq R_\lambda} (1+|u|)\|_{L^m_{t,x}}^m 
 + \|\mathbf{1}_{|\tu|\geq R_\lambda} (1+|\tu|)\|_{L^m_{t,x}}^m\right) \\
 & \quad \quad + C \left(\mathcal{G}_\alpha(\delta, \eps, \lambda) + \delta^{-1} d(\sigma,\tsigma)\right) \EE\left(1+\|u\|_{L^{m+1}_{t,x}}^{m+1}+\|\tu\|_{L^{m+1}_{t,x}}^{m+1} \right)
\end{equs}
where $R_\lambda = \sup\{R\in[0,\infty]: |\fra(r) - \tfra(r)| \leq \lambda, \text{ for every } |r|\leq R\}$ and $\mathcal{G}_\alpha$ as in \eqref{eq:def-G}. 
\end{proposition}

\begin{proof} The proof is similar to \cite[proof of Theorem 4.1]{DGG19}. The main difference lies in the presence of $\psi$ since we impose Dirichlet
boundary conditions, thus, we cannot let $\psi = 1$ as in the case of periodic boundary conditions dealt in \cite{DGG19}. We only give a sketch of the proof,
highlighting the differences. 

Let $\eta_\delta$ be a symmetric smooth approximation of $|\cdot|$ given by 
\begin{equs}
 \eta_\delta(0) = \eta_\delta^{\prime}(0) = 0, \, \eta_\delta^{\prime\prime}(r) = \delta^{-1} \tilde\eta(\delta^{-1}|r|)
\end{equs}
for some non-negative $\tilde\eta\in C^\infty(\RR)$ which is bounded by $2$, supported in $(0,1)$ and integrates to $1$. Below we repeatedly use the following
properties of $\eta_\delta$,
\begin{equs}
 |\eta_\delta(r) - |r|| \lesssim \delta, \, \supp \eta_\delta^{\prime\prime} \subset [-\delta,\delta], \,\int |\eta^{\prime \prime}_\delta(r-\tilde r)|\dd \tilde r \leq 2,
 \, |\eta_\delta^{\prime\prime}(r)|\leq 2\delta^{-1}.  
\end{equs}
For $y \in Q$, $s \in (0,T)$ and $\eps,\theta>0$ sufficiently small, we also set 
\begin{align*}
 & \phi_{\eps,\theta}(t,x,s,y) := \rho_\theta(t-s) \varrho_\eps(x-y) \varphi\left(\frac{t+s}{2}\right) 
 \psi(x), \, \phi_\eps(t,x,y) := \varrho_\eps(x-y) \varphi(t) \psi(x).
\end{align*}

We first apply the definition of entropy solutions with $u=u(t,x)$, $\eta(u) = \eta_\delta(u - a)$, for $ a \in \bR$, and 
$\phi(t,x) = \phi_{\eps,\theta}(t,x,s,y)$. Noting that $\phi_{\eps,\theta}(0,x,s,y) =0$ for $\theta$ sufficiently small, this 
gives that, $\PP$-almost surely,
\begin{equs}
 -\int_0^T \int_Q \eta_\delta(u-a) \partial_t \phi_{\eps,\theta} \dd x \dd t  
 & \leq \int_0^T \int_Q [\eta'_\delta(\cdot-a) \fra ^2,a](u) \Delta_x\phi_{\eps,\theta} 
 + \frac{1}{2}  \phi_{\eps,\theta} \eta_\delta''(u - a) |\sigma(x,u)|^2_{\ell^2} \dd x \dd t  \\ 
 & \quad - \int_0^T \int_Q \phi_{\eps,\theta} \eta_\delta''(u - a) |\nabla_x[\fra ](u)|^2 \dd x \dd t \\
 & \quad + \int_0^T \int_Q   \phi_{\eps,\theta} \eta_\delta'(u - a) \sigma^k(x,u)  \dd x \dd\beta^k(t).
\end{equs}
We now plug in $ \tilde u(s,y)$ in place of $a$ (all the expressions are smooth functions of $a$) and integrate over $s,y$. Then we repeat the same procedure with  the roles of $u$ and $\tilde u$ reversed, we add the two resulting inequalities and
take expectations to obtain the estimate
\begin{equs} 
 & - \EE \int_{t,x,s,y} \eta_\delta(u- \tilde u)  \left(\partial_t\phi_{\eps,\theta}+ \partial_s\phi_{\eps,\theta}\right)  
 \\
 & \quad \leq \EE \int_{t,x,s,y} \left([\eta_\delta^{\prime}(\cdot - \tu) \fra ^2,\tu](u) \Delta_x \phi_{\eps,\theta}+[\eta^{\prime}_\delta(\cdot- u) \tfra ^2,u](\tu) \Delta_y\phi_{\eps,\theta} \right) 
 \\ 
 & \quad \quad - \EE \int_{t,x,s,y} \eta_\delta^{\prime \prime}(u - \tu)\left(  |\nabla_x [\fra](u)|^2 + |\nabla_y [\tfra](\tu)|^2 \right) \phi_{\eps,\theta} 
 \\
 & \quad \quad + \EE \int_{s,y} \left[ \int_{t,x} \eta_\delta^{\prime}(u - a) \sigma^k(x,u) \phi_{\eps,\theta} \dd \beta^k(t)\right]_{a=\tu}  + \EE \int_{t,x} \left[ \int_{s,y} \eta_\delta^{\prime}(\tu - a) \tsigma^k(y,\tu) \phi_{\eps,\theta} \dd \beta^k(s) \right]_{a=u}
 \\
 & \quad \quad + \EE \int_{t,x,s,y} \frac{1}{2}\eta_\delta^{\prime \prime}(u - \tu)\phi_{\eps,\theta} \left(  |\sigma(u)|^2_{\ell^2} + |\tsigma(\tu)|^2_{\ell^2}\right) \label{eq:int_doubling} 
\end{equs}
The next step is to pass to the limit $\theta\to 0$ to obtain the estimate
\begin{align} 
 \nonumber
 - \EE \int_{t,x,y} \eta_\delta(u - \tu) \partial_t\phi_\eps & \leq \EE \int_{t,x,y}[\eta_\delta^{\prime}(\cdot - \tu) \fra^2,\tu] (u)  
 \Delta_x \phi_\eps + [\eta_\delta^{\prime}(\cdot - u) \tfra^2,u](\tu) \Delta_y\phi_\eps   && (=:I_1+I_2)\\   
 \nonumber
 & \quad - \EE \int_{t,x,y}\eta_\delta^{\prime \prime}(u - \tu) \phi_\eps  \left(|\nabla_x [\fra](u)|^2 +|\nabla_y [\tfra](\tu)|^2 \right) &&  (=:I_3)\\
 & \quad + \frac{1}{2} \EE \int_{t,x,y} \eta_\delta^{\prime \prime}(u - \tu)\phi_\eps |\sigma(x,u)-\tsigma(y,\tu)|^2_{\ell^2}, &&(=:I_4) \label{eq:int_doubling_theta=0} 
\end{align}
with $u=u(t,x)$ and $\tu = \tu(t,y)$. 
To do so, we first notice that  
\begin{align*}
 \partial_t\phi_{\eps,\theta}(t,x,s,y) + \partial_s\phi_{\eps,\theta}(t,x,s,y) = \rho_\theta(t-s) \varrho_\eps(x-y) (\D_t \varphi) \left(\frac{t+s}{2}\right)
 \psi(x). 
\end{align*}
We then use \cite[Proposition 3.5]{DGG19} to pass to the limit $\theta\to 0$ for each term in \eqref{eq:int_doubling}.
The main difference here is the presence of $\psi$, but it is easy to see that \cite[Proposition 3.5]{DGG19} still applies in our case since $\psi$ is in $\CC^\infty_c(Q)$.
The only subtle terms in \eqref{eq:int_doubling} are those involving the stochastic integrals, but they can also be treated as in \cite[proof of Theorem 4.1]{DGG19}. For the 
first stochastic integral we notice that $\phi_{\eps,\theta}$ vanishes for $t\notin [s,s+\theta]$, and since $\tilde u(s,y)$ is $\mathcal{F}_s$-measurable,
the expectation vanishes for every $\theta>0$ by a simple factorisation argument. For the second stochastic integral, we use the $(\star)$-property for $u(t,x)$ which together 
with the last term in \eqref{eq:int_doubling} gives $I_4$ in \eqref{eq:int_doubling_theta=0} if we let $\theta\to 0$. For the term $I_1$ we notice that 
\begin{align*}
 I_1 & = \EE \int_{t,x,y}[\eta_\delta^{\prime}(\cdot - \tu) \fra^2,\tu](u)  
 \Delta_x \phi_\eps \\
 & = -\EE \int_{t,x,y}[\eta_\delta^{\prime}(\cdot - \tu) \fra^2,\tu](u) \D^2_{x_iy_i} \phi_\eps && (=: I_{1,1}) \\
 & \quad + \EE \int_{t,x,y} [\eta_\delta^{\prime}(\cdot - \tu) \fra^2,\tu](u)   \varphi(t) 
 \partial_{x_i} \left(\partial_{x_i}\psi(x) \varrho_\eps(x-y)\right). && (=: I_{1,2})
\end{align*}
We furthermore notice that
\begin{equs}
 I_{1,1} &
 = - \EE \int_{t,x,y} \int_{\tu}^{u} \int_{\tu}^{r} \eta_\delta''(r - \tr) \mathfrak{a}(r)^2  \dd {\tr} \dd r \, \partial_{x_iy_i}^2 \phi_\eps \\
 & = - \EE \int_{\tu\leq u} \int_{\tu}^{u} \int_{\tu}^{u} 
 \mathbf{1}_{\{\tr \leq {r}\}} \eta_\delta''(r - {\tr}) \mathfrak{a}(r)^2 \dd {\tr} \dd r \, \partial_{x_iy_i}^2 \phi_\eps\\
 & \quad - \EE \int_{\tu\geq u} \int_{\tu}^{u} \int_{\tu}^{u} 
 \mathbf{1}_{\{\tr\geq r\}} \eta_\delta''(r - \tr) \mathfrak{a}(r)^2  \dd {\tr} \dd r \, \partial_{x_iy_i}^2 \phi_\eps.
\end{equs}
Similarly we have that $I_2 = I_{2,1} + I_{2,2}$ where 
\begin{equs}
 I_{2,1} & := - \EE \int_{\tu\leq  u} \int_{\tu}^{ u} \int_{\tu}^{ u} 
 \mathbf{1}_{\{\tr \leq r \}} \eta_\delta^{\prime \prime}(r - \tr) \tfra(\tr)^2 \dd r \dd \tr \, \partial_{y_ix_i}^2 \phi_\eps\\
 & \quad - \EE \int_{\tu\geq  u} \int_{\tu}^{ u} \int_{\tu}^{ u} 
 \mathbf{1}_{\{\tr \geq r \}} \eta_\delta^{\prime \prime}(r - \tr) \tfra(\tr)^2 \dd r \dd \tr \, \partial_{y_ix_i}^2 \phi_\eps
\end{equs}
and
\begin{equs}
 I_{2,2} & := \EE \int_{t,x,y} [\eta_\delta^{\prime}(\cdot - u) \tfra, u](\tu) \varphi(t) \partial_{y_i}\left(\partial_{x_i} 
 \psi(x) \varrho_\eps(x-y)\right).
\end{equs}
For $I_3$, as in \cite[proof of Theorem 4.1]{DGG19}, we have the bound
\begin{equs}
  I_3 & \leq  2 \EE \int_{\tu\leq  u} \int_{\tu}^{ u} \int_{\tu}^{ u} \mathbf{1}_{\{\tr\leq r\}} \eta_\delta^{\prime \prime}(r - \tr) \tfra(\tr) \fra(r)  \dd \tr \dd r \,  \partial_{x_iy_i}^2\phi_\eps \\
 & \quad + 2 \EE \int_{\tu\geq  u} \int_{\tu}^{ u} \int_{\tu}^{ u}\mathbf{1}_{\{\tr\geq r\}} \eta_\delta^{\prime \prime}(r - \tr) \tfra(\tr) \fra(r )  \dd \tr \dd r \, \partial_{x_iy_i}^2\phi_\eps.
\end{equs}
We now add the terms $I_{1,1}$, $I_{2,1}$ and $I_3$ to obtain the estimate
\begin{equs}
 & I_{1,1} + I_{2,1} + I_3\leq \EE \int_{t,x,y} \int_{{\tu}}^{ u} \int_{{\tu}}^{ u} \eta_\delta^{\prime \prime}(r-\tr) 
 |\fra(r) - \tfra(\tr)|^2 \dd \tr \dd r|\partial^2_{x_iy_i}\phi_\eps|. 
\end{equs}
Altogether, the previous estimates imply the bound
\begin{align}
 - \EE \int_{t,x,y} \eta_\delta(u - \tu) \partial_t\phi_\eps & \leq \EE \int_{t,x,y} [\eta_\delta^{\prime}(\cdot - \tu) \fra^2,\tu](u)   \varphi(t) 
 \partial_{x_i} \left(\partial_{x_i}\psi(x) \varrho_\eps(x-y)\right) && (=:I_1') \nonumber 
 \\   
 & \quad + \EE \int_{t,x,y} [\eta_\delta^{\prime}(\cdot - u) \tfra^2,u](\tu) \varphi(t) \partial_{y_i}\left(\partial_{x_i} 
 \psi(x) \varrho_\eps(x-y)\right) && (=:I_2') \nonumber
 \\
 & \quad + \EE \int_{t,x,y} \int_{{\tu}}^{ u} \int_{{\tu}}^{ u} \eta_\delta''(r-\tr) 
 |\fra(r) - \tfra(\tr)|^2 \dd \tr \dd r |\partial^2_{x_iy_i}\phi_\eps| && (=:I_3') \nonumber 
 \\
 & \quad + \EE \int_{t,x,y} \frac{1}{2}\eta_\delta^{\prime \prime}(u - \tu)\phi_\eps |\sigma(x,u)-\tsigma(y,\tu)|^2_{\ell^2}. && (=:I_4')
 \label{eq:int_doubling_theta=0_grouped}
\end{align}

For the term on the left hand side of \eqref{eq:int_doubling_theta=0_grouped} we have that
\begin{equs}
 \left|\int_{t,x,y} (\eta_\delta(u- \tu) - |u - \tilde u|) \partial_t\phi_\eps \right| \lesssim \delta,
\end{equs}
since $|\eta_\delta(\cdot) - |\cdot||\lesssim \delta$. 

For the terms $I_1'$ and $I_2'$ we have that 
\begin{align*}
 I_1' & =  \EE \int_{t,x,y}[\eta_\delta^{\prime}(\cdot - \tu) \fra^2,\tu](u)   \varphi(t) \Delta \psi(x) \varrho_\eps(x-y) && (=:I_{1,1}')  \\
 & \quad + \EE \int_{t,x,y} [\eta_\delta^{\prime}(\cdot - \tu) \fra^2,\tu](u)   \varphi(t) \partial_{x_i}\psi(x)
 \partial_{x_i}(\varrho_\eps(x-y)) && (=:I_{1,2}')
\end{align*}
and
\begin{equs}
 I_2' & =\EE \int_{t,x,y} [\eta_\delta^{\prime}(\cdot - u) \tfra^2,u](\tu) \varphi(t) \partial_{x_i}\psi(x) 
 \partial_{y_i}(\varrho_\eps(x-y)).
\end{equs}

The term $I_{1,1}'$ can be written as
\begin{align*}
 I_{1,1}' & = \EE \int_{t,x,y} \int_{\tu}^{ u} \sgn(r - \tu) \fra(r)^2 \dd r \, \varphi(t) \Delta
 \psi(x) \varrho_\eps(x-y) \\
 & \quad + \EE \int_{t,x,y} \int_{\tu}^{ u} \left(\eta_\delta'(r- \tu)-\sgn(r - \tu)\right) \fra(r)^2 \dd r \, \varphi(t) \Delta
 \psi(x) \varrho_\eps(x-y) \\
 & = \EE \int_{t,x,y} \int_{\tu}^{ u} \sgn(r- \tu) \fra(r)^2 \dd r \,  \varphi(t)  
 \Delta
 \psi(x) \varrho_\eps(x-y) && (=:I_{1,1,1}') \\
 & \quad + \EE \int_{t,x,y} \int_{\tu}^{ u} \mathbf{1}_{|r- \tu|\leq \delta}\left(\eta_\delta'(r- \tu)-\sgn(r - \tu)\right) \fra(r)^2 \dd r \varphi(t) \Delta
 \psi(x) \varrho_\eps(x-y). && (=:I_{1,1,2}')
\end{align*}
For the term $I_{1,1,2}'$ using the boundedness of $ \Delta\psi$, the fact that $\int_x \varrho_\eps(x-y) \lesssim 1$ 
and Assumption \ref{as:A}-\ref{as:A first} we get that
\begin{equs}
 |I_{1,1,2}'| & \lesssim \EE \int_{t,x,y} \int_{|r -\tu|\leq \delta} \fra(r)^2 \dd r \varphi(t) \left| \Delta
 \psi(x)\right| \varrho_\eps(x-y)  \le \delta \, \EE \int_{t,x,y} \sup_{|r-\tu|\leq \delta} \fra(r)^2 \, \varphi(t) \varrho_\eps(x-y) \\
 &  \le  \delta \ \EE \int_{t,y} \left(\int_0^{\delta+|\tu |} \fra'(r) \dd r \right)^2 \varphi(t) \le \delta  \ \EE\left(1 + \|\tilde u\|_{L^m_{t,x}}^m\right).
 \end{equs}


The term $I_{1,2}'$ can be written as
\begin{align*}
 I_{1,2}'& = \EE \int_{t,x,y} \int_{\tu}^{ u} \sgn(r - \tu)\fra(r)^2 \dd r \varphi(t) 
 \partial_{x_i}\psi(x) \partial_{x_i} (\varrho_\eps(x-y)) \\
 & \quad + \EE \int_{t,x,y} \int_{\tu}^{ u} \left(\eta_\delta'(r - \tu) - \sgn(r - \tu\right))\fra(r)^2
 \dd r \varphi(t) \partial_{x_i}\psi(x)\partial_{x_i}(\varrho_\eps(x-y)) \\
 & = \EE \int_{t,x,y} \int_{\tu}^{ u} \sgn(r - \tu)\fra(r)^2 \dd r \varphi(t) 
 \partial_{x_i}\psi(x) \partial_{x_i} (\varrho_\eps(x-y)) && (=:I_{1,2,1}') \\
 & \quad + \EE \int_{t,x,y} \int_{\tu}^{ u} \mathbf{1}_{|r - \tu|\leq \delta} \left(\eta_\delta'(r - \tu) - \sgn(r - \tu)\right) \mathfrak{a}(r)^2
 \dd r \varphi(t) \partial_{x_i}\psi(x) \partial_{x_i}(\varrho_\eps(x-y)) && (=:I_{1,2,2}').
\end{align*}
Similar calculations as in the case of  $I_{1,1,2}'$, but now using that $\int_x |\partial_{x_i}(\varrho_\eps(x-y))]| \lesssim \eps^{-1}$
imply that
\begin{equs}
 & I_{1,2,2}^{\prime} \le \delta \eps^{-1} \EE\left(1 + \|\tilde u\|_{L^m_{t,x}}^m\right).
\end{equs}
Putting these estimates together gives 
\begin{equs}
 I^{\prime}_1 & \leq \EE \int_{t,x,y} \int_{\tu}^{ u} \sgn(r- \tu) \fra(r)^2\dd r \, \varphi(t)  
 \Delta\psi(x) \varrho_\eps(x-y) \\
 & \quad + \EE \int_{t,x,y} \int_{\tu}^{ u} \sgn(r- \tu) \fra(r)^2\dd r \, \varphi(t) 
 \partial_{x_i}\psi(x) \partial_{x_i} (\varrho_\eps(x-y)) \nonumber + C \delta \eps^{-1} \EE\left(1 + \|\tilde u\|_{L^m_{t,x}}^m\right). 
 \label{eq:estimate-I_1'}
\end{equs}
We also have
\begin{align*}
 I_{2}^{\prime} & = \EE \int_{t,x,y} \int_{ u}^{\tu} \sgn(\tr-  u) \tfra(\tr)^2\dd \tr \, \varphi(t) 
 \partial_{x_i}\psi(x) \partial_{y_i} (\varrho_\eps(x-y)) && (=:I_{2,1}^{\prime})\\
 & \quad + \EE \int_{t,x,y} \int_{ u}^{\tu} \mathbf{1}_{|\tr-  u|\leq \delta} \left(\eta_\delta^{\prime}(\tr-  u) - \sgn(\tr-  u) \right) \tfra(\tr)^2
 \dd\tr \, \varphi(t) \partial_{x_i}\psi(x)\partial_{y_i}(\varrho_\eps(x-y)) && (:=I_{2,2}^{\prime})
\end{align*}
and as before
%
\begin{equs}
 I_2' & \leq \ \EE \int_{t,x,y} \int_{u}^{\tu} \sgn(\tr - u) \tfra(\tr)^2 \dd \tr \,  \varphi(t) 
 \partial_{x_i}\psi(x) \partial_{y_i} (\varrho_\eps(x-y)) + C \delta \eps^{-1} \EE\left(1 + \|u\|_{L^m_{t,x}}^m\right).  \label{eq:estimate-I_2'}
\end{equs}
The term $I_3'$  can be treated exactly as in 
\cite[proof of Theorem 4.1]{DGG19}, to obtain
\begin{equs}
 I_3' & \lesssim \eps^{-2} (\delta^{2\al}+\lambda^2) \EE\left(1+\|u\|_{L^m_{t,x}}^m + \|\tilde u\|_{L^m_{t,x}}^m\right) \\
 & \quad + \eps^{-2} \left(\EE\|\mathbf{1}_{|u|\geq R_\lambda} (1+|u|)\|_{L^m_{t,x}}^m
 + \EE\|\mathbf{1}_{|\tilde u|\geq R_\lambda} (1+|\tilde u|)\|_{L^m_{t,x}}^m\right). \label{eq:estimate-I_3'}
\end{equs}
For the term $I_4'$ we notice that
\begin{equs}
 I_4' & \le \EE \int_{t,x,y} \eta_\delta''( u  -  {\tu} )\phi_\eps\left(  |\sigma(x, u ) - \sigma(x, {\tu} )|^2_{\ell^2} + |\sigma(x, u ) - \sigma(y, {\tu} )|^2_{\ell^2} +|\sigma(y, {\tu} ) - \tsigma(y, {\tu} )|^2_{\ell^2}\right) 
 \\
 & \lesssim \left(\delta^{2\kappa} + \delta^{-1} \eps^{2\bar\kappa} + \delta^{-1} d(\sigma,\tsigma)\right)
 \EE\left(1 + \|u\|_{L^{m+1}_{t,x}}^{m+1} + \|\tilde u\|_{L^{m+1}_{t,x}}^{m+1}\right)\label{eq:estimate-I_4'}     
\end{equs}
where we use Assumption \ref{as:noise} and the fact that $|\eta_\delta''|\lesssim \delta^{-1}$.

Noting that $\partial_{y_i} (\varrho_\eps(x-y)) = - \partial_{x_i} (\varrho_\eps(x-y))$ and that for every $u, \tilde u$ and 
non-decreasing differentiable function $A$ we have the identity
\begin{equs}
 \int_{\tilde u}^u \sgn(r - \tilde u) A'(r) \, dr = |A(u) - A(\tilde u)|,
\end{equs}
we obtain by \eqref{eq:estimate-I_1'}-\eqref{eq:estimate-I_4'} (and the fact that $\delta \leq \delta^{2\kappa}$)
\begin{equs}
 & -\EE\int_0^T \int_{Q^2} | u- \tu| \partial_t\varphi(t) \psi(x) \varrho_\eps(x-y) \dd x \dd y \dd t  
 \\
 & \quad \leq \EE \int_0^T \int_{Q^2} |A( u) - A( \tu)| \varphi(t) \Delta \psi(x)\varrho_\eps(x-y) \dd x \dd y \dd t  
 \\
 & \quad \quad + \EE \int_0^T \int_{Q^2} |A( u) - A( \tu)| \varphi(t) \partial_{x_i}\psi(x) \partial_{x_i}(\varrho_\eps(x-y)) \dd x \dd y \dd t 
 \\ 
 & \quad \quad - \EE \int_0^T \int_{Q^2} |\tA( u) -\tA( \tu)| \varphi(t) \partial_{x_i}\psi(x) \partial_{x_i}(\varrho_\eps(x-y)) \dd x \dd y \dd t
 \\
 & \quad \quad + C \eps^{-2} \EE\left(\|\mathbf{1}_{|u|\geq R_\lambda} (1+|u|)\|_{L^m_{t,x}}^m 
 + \|\mathbf{1}_{|\tilde u|\geq R_\lambda} (1+|\tilde u|)\|_{L^m_{t,x}}^m\right) 
 \\
 & \quad \quad + C \left(\mathcal{G}_\alpha(\delta, \eps, \lambda) + \delta^{-1} d(\sigma,\tsigma)\right)  \EE\left(1+\|u\|_{L^{m+1}_{t,x}}^{m+1}+\|\tilde u\|_{L^{m+1}_{t,x}}^{m+1} \right). \label{eq:almost-done}             
\end{equs}
By the triangle inequality we have that
\begin{equs}
 & \EE \int_0^T \int_{Q^2} ( |A(u) - A( \tu)|-|\tA(u) - \tA( \tu)|)   \varphi(t) \partial_{x_i}\psi(x) \partial_{x_i}(\varrho_\eps(x-y)) \dd x \dd y \dd t  
 \\
 & \quad \leq \EE \int_0^T \int_{Q^2} (|A(u) - \tA(u)|+|A( \tu) - \tA( \tu)|)  \varphi(t) \left|\partial_{x_i}\psi(x) \partial_{x_i}(\varrho_\eps(x-y))\right|  \dd x \dd y \dd t 
 \\
 & \quad \lesssim \eps^{-1} \lambda \EE\left(\|u\|_{L^{\frac{m+1}{2}}_{t,x}}^{\frac{m+1}{2}} + \|\tilde u\|_{L^{\frac{m+1}{2}}_{t,x}}^{\frac{m+1}{2}}\right)
 + \eps^{-1} \EE\left( \|\mathbf{1}_{|u|\geq R_\lambda}u\|_{L^m_{t,x}}^m + \|\mathbf{1}_{|\tilde u|\geq R_\lambda}\tilde u\|_{L^m_{t,x}}^m\right) \label{eq:almost-done-2}
\end{equs}
where in the last step we use that $\int_y |\partial_{x_i}(\varrho_\eps(x-y))|,\int_x |\partial_{x_i}(\varrho_\eps(x-y))| \lesssim \eps^{-1}$ and that for every $r\in \RR$, by Assumption \ref{as:A}-\ref{as:A first},
\begin{align*}
 |A(r) - \tA(r)| & \leq  \int_0^{|r|} |\fra(\zeta)^2 - \tfra(\zeta)^2| \dd  \zeta \lesssim  \lambda \int_0^{|r|} \left(|\fra(\zeta)| + |\tfra(\zeta)| \right) \dd \zeta + \mathbf{1}_{|r|\geq R_\lambda} 
 \int_0^{|r|}\left(\fra(\zeta)^2 + \tfra(\zeta)^2\right) \dd \zeta \\
 & \lesssim \lambda |r|^{\frac{m+1}{2}} + \mathbf{1}_{|r|\geq R_\lambda} |r|^m.
\end{align*}
Similarly, since $\int_x |\varrho_\eps(x-y)| \lesssim 1$, we see that
\begin{equs}
 & \EE \int_0^T \int_{Q^2} \left(|A( u) - A( \tu)| - |A( u) - \tA(\tu)|\right) \varphi(t) \Delta \psi(x)\varrho_\eps(x-y) \dd x \dd y \dd t \\ 
 & \quad \lesssim \lambda \EE\left(\|\tilde u\|_{L^{\frac{m+1}{2}}_{t,x}}^{\frac{m+1}{2}}\right)
 + \EE\left(\|\mathbf{1}_{|\tilde u|\geq R_\lambda}\tilde u\|_{L^m_{t,x}}^m\right). \label{eq:almost-done-3}
\end{equs}
By \eqref{eq:almost-done}-\eqref{eq:almost-done-3} we obtain the desired inequality. 
\end{proof}

The next lemma is a point-wise in time version of Proposition \ref{prop:basic_est}.

\begin{lemma} \label{lem:almost_final} Let $u, \tu$ be entropy solutions of the Dirichlet problems $\mathcal{E}(A,\sigma,\xi)$, $\mathcal{E}(\tA,\tsigma,\txi)$, 
where the data satisfy Assumptions \ref{as:A} and \ref{as:noise}, and assume that $u$ satisfies the $(\star)$-property with coefficient $\sigma$ (see Definition \ref{def:star_prop}).
Then, for every $\psi \in C^\infty_c(Q)$ and $\al\in(0,\frac{1}{2})$ there exists $C\equiv C(\mathcal{S}, \alpha, \psi)>0$ such that for every
$\lambda,\eps,\delta \in (0,1)$ and every right Lebesgue points $s\leq t$ of the mapping
\begin{equation}
 \tau \mapsto \EE \int_{x,y} |u(\tau,x) - \tilde u(\tau,y)| \psi(x) \varrho_\eps(x-y). \label{eq:leb_points}
\end{equation}
we have that
\begin{equs}
 & \EE \int_{Q^2} |u(t,x) - \tu(t,y)| \psi(x) \varrho_\eps(x-y) \dd x \dd y \\
 & \quad \leq \EE \int_{Q^2} |u(s,x) - \tu(s,y)| \psi(x) \varrho_\eps(x-y) \dd x \dd y
 \\
 & \quad \quad + \EE \int_s^t \int_{Q^2} |A(u(\tau,x)) - \tA(\tu(\tau,y))| \Delta\psi(x) \varrho_\eps(x-y)\dd x \dd y \dd \tau 
 \\
 & \quad \quad + C \eps^{-2} \EE\left(\|\mathbf{1}_{|u|\geq R_\lambda} (1+|u|)\|_{L^m_{t,x}}^m 
 + \|\mathbf{1}_{|\tu|\geq R_\lambda} (1+|\tu|)\|_{L^m_{t,x}}^m\right) 
 \\
 & \quad \quad + C \left(\mathcal{G}_\alpha(\delta, \eps, \lambda) + \delta^{-1} d(\sigma,\tsigma)\right) 
 \EE\left(1+\|u\|_{L^{m+1}_{t,x}}^{m+1}+\|\tu\|_{L^{m+1}_{t,x}}^{m+1}\right), \label{eq:almost_final}
\end{equs}
where $R_\lambda = \sup\{R\in[0,\infty]: |\fra(r) - \tfra(r)| \leq \lambda, \text{ for every } |r|\leq R\}$ and $\mathcal{G}_\alpha$ 
as in \eqref{eq:def-G}.
\end{lemma}

\begin{proof} The proof is given in \cite[Proof of Theorem 4.1]{DGG19}, by approximating the function $\mathbf{1}_{[s,t]}$ and using Proposition \ref{prop:basic_est}.
\end{proof}

We are now ready to prove the main result of this section.

\begin{lemma} \label{lem:final_psi}  Let $u, \tu$ be entropy solutions of the Dirichlet problems $\mathcal{E}(A,\sigma,\xi)$, $\mathcal{E}(\tA,\tsigma,\txi)$, 
where the data satisfy  Assumptions \ref{as:A} and \ref{as:noise}, and assume that $u$ satisfies the $(\star)$-property with coefficient $\sigma$ (see Definition \ref{def:star_prop}).
\begin{enumerate}
 \item \label{it:final_psi_1} For every $\psi \in C^\infty_c(Q)$ and $\al\in(0,\frac{1}{2})$ there exists $C\equiv C(\mathcal{S}, \alpha, \psi)>0$ such that for every
 $\lambda,\delta \in (0,1)$ and $\eps\in(0,1)$ sufficiently small
 \begin{align*}
  & \EE\int_0^T \int_Q |u(t,x) - \tu(t,x)| \psi(x) \dd x \dd t \\
  & \quad \leq T \EE\int_Q |\xi(x) - \txi(x)| \psi(x) \dd x + T \sup_{|h|\leq \eps}\EE \int_Q |\txi(x) - \txi(x+h)| \psi(x) \dd x \\
  & \quad \quad + \EE \int_0^T \int_0^t \int_Q |A(u(\tau,x))- \tA(\tu(\tau,x))| \Delta\psi(x)\dd x \dd \tau \dd t 
  + C \eps^{\frac{2}{m+1}} \EE\left( 1 + \|\nabla [\tfra](\tu)\|_{L^1_{t,x}}\right) \\
  & \quad \quad  + C \eps \EE \|\nabla \tA(\tu)\|_{L^1_{t,x}} + C \eps^{-2} \EE\left(\|\mathbf{1}_{|u|\geq R_\lambda} (1+|u|)\|_{L^m_{t,x}}^m 
  + \|\mathbf{1}_{|\tu|\geq R_\lambda} (1+|\tu|)\|_{L^m_{t,x}}^m\right) 
  \\
  & \quad \quad + C \left(\mathcal{G}_\alpha(\delta, \eps, \lambda) + \delta^{-1} d(\sigma,\tsigma)\right) 
  \EE\left(1+\|u\|_{L^{m+1}_{t,x}}^{m+1}+\|\tu\|_{L^{m+1}_{t,x}}^{m+1}\right).
 \end{align*}
 where $R_\lambda = \sup\{R\in[0,\infty]: |\fra(r) - \tfra(r)| \leq \lambda, \text{ for every } |r|\leq R\}$ and $\mathcal{G}_\alpha$ as in
 \eqref{eq:def-G}.
 \item \label{it:final_psi_2} If we furthermore assume that $\tA = A$ and $\tsigma = \sigma$, then for every $\psi\in C_c^\infty(Q)$ and almost every 
 $s<t\leq T$ we have that
 \begin{align*}
  & \EE \int_Q |u(t,x) - \tilde u(t,x)| \psi(x) \dd x - \EE \int_Q |u(s,x) - \tilde u(s,x)| \psi(x) \dd x \\
  & \quad \leq \EE\int_s^t \int_Q |A(u(\tau,x)) - A(\tu(\tau,x))| \Delta \psi(x) \dd x \dd \tau.
 \end{align*}
 In addition, this estimate holds for $s=0$ and $u(s)$, $\tu(s)$ replaced by $\xi$, $\txi$. 
\end{enumerate}
\end{lemma}

\begin{proof}[Proof of Lemma \ref{lem:final_psi}-\ref{it:final_psi_1}] We first notice that for all $\varepsilon \in(0,1)$ sufficiently small by the mean value
theorem for $\tA(\tu(\tau,\cdot))$ we have that 
\begin{equs}
 & \EE \int_s^t \int_{Q^2} |A(u(\tau,x)) - \tA(\tu(\tau,y))| \Delta\psi(x) \varrho_\eps(x-y)\dd x \dd y \dd \tau \\
 & \quad \leq \EE \int_s^t \int_Q |A(u(\tau,x)) - \tA(\tu(\tau,x))| \Delta\psi(x) \dd x \dd \tau
 + C \eps \EE \|\nabla \tA(\tu)\|_{L^1_{t,x}} \label{eq:grad_est}
\end{equs}
for some $C\equiv C(\mathcal{S},\|\Delta\psi\|_{L^\infty_x})>0$, where we also use that $\int_y |(x-y)\varrho_\eps(x-y)| \lesssim \eps$. Then, by Lemma \ref{lem:almost_final} we have that for every right Lebesgue points $s\leq t$ of the mapping \eqref{eq:leb_points}
\begin{equs}
 \EE \int_{Q^2} |u(t,x) - \tilde u(t,y)| \psi(x)  \varrho_\eps(x-y) \dd x \dd y &
 \leq \EE \int_{Q^2} |u(s,x) - \tilde u(s,y)| \psi(x) \varrho_\eps(x-y) \dd x \dd y \\ 
 & \quad + \EE \int_s^t \int_Q |A(u(\tau,x)) - \tA(\tu(\tau,x))| \Delta\psi(x)\dd x \dd \tau + M, \\ 
 \label{eq:almost_final_M}
\end{equs}
where we write $M\equiv M(\delta,\eps,\lambda)>0$ for the remaining terms on the right hand side of \eqref{eq:almost_final} plus 
$C \eps \EE \|\nabla \tA(\tu)\|_{L^1_{t,x}}$. As in \cite[Lemma 3.2]{DGG19} we can prove that
\begin{equation} \label{eq:ic-is-Lebesgue}
 \lim_{h \to 0} \frac{1}{h} \int_0^h \EE\int_Q |u(s,x) - \xi(x)|^2 \psi(x)^2 \dd x \dd s=0,
\end{equation}
and similarly for $\tilde u$, $\tilde \xi$. Hence, if we integrate  \eqref{eq:almost_final_M} over $s\in(0,h)$, divide by $h$ and let $h \to 0$ we get 
\begin{equs}
 \EE \int_{Q^2} |u(t,x) - \tilde u(t,y)| \psi(x)  \varrho_\eps(x-y) \dd x \dd y &
 \leq \EE \int_{Q^2} |\xi(x) - \txi (y)| \psi(x) \varrho_\eps(x-y) \dd x \dd y \\ 
 & \quad + \EE \int_0^t \int_Q |A(u(\tau,x)) - \tA(\tu(\tau,x))| \Delta\psi(x)\dd x \dd \tau + M. \\ 
\end{equs}
We now integrate the above inequality over $t\in[0,T]$ to obtain
\begin{align*}
 & \EE \int_0^T \int_{Q^2} |u(t,x) - \tilde u(t,y)| \psi(x)  \varrho_\eps(x-y) \dd x \dd y \dd t \\
 & \quad \leq T \EE \int_{Q^2} |\xi(x) - \txi(y)| \psi(x) \varrho_\eps(x-y) \dd x \dd y \\
 & \quad \quad +  \EE \int_0^T \int_0^t  \int_Q |A(u(\tau,x)) - \tA(\tu(\tau,x))| \Delta\psi(x)\dd x \dd \tau \dd t + T M.
\end{align*}
Moreover, for every $\eps\in(0,1)$ sufficiently small such that $x+h\in Q$ whenever $x\in \supp \psi$ and $|h|\leq \eps$, we have that
\begin{align*}
 \EE \int_{Q^2} |\xi(x) - \txi(y)| \psi(x) \varrho_\eps(x-y) \dd x \dd y & \leq \sup_{|h|\leq \eps} \EE\int_Q |\txi(x) - \txi(x+h)| \psi(x) \dd x \\ 
 & \quad + \EE \int_Q |\xi(x) - \txi(x)| \psi(x) \dd x.
\end{align*}
We finally notice that by \cite[Lemma 3.1]{DGG19} 
\begin{align*}
 & \EE \left|\int_0^T\int_{Q^2} |u(t,x) - \tilde u(t,y)| \psi(x) \varrho_\eps(x-y) \dd x \dd y \dd t - \int_0^T \int_Q |u(t,x) - \tilde u(t,x)| \psi(x) \dd x \dd t\right| \\
 & \quad \lesssim \eps^{\frac{2}{m+1}} \left(1+\EE \|\nabla [\tfra](\tilde u)\|_{L^1_{t,x}}\right).
\end{align*}
This implies the assertion.
\end{proof}

\begin{proof}[Proof of Lemma \ref{lem:final_psi}-\ref{it:final_psi_2}] We first notice that since $\tA=A$ we can choose $\lambda =0$ and $R_\lambda=\infty$ in \eqref{eq:almost_final}. Since we also have that
$d(\sigma,\tsigma)=0$, using again \eqref{eq:grad_est}, \eqref{eq:almost_final} reads as follows, 
\begin{align*}
 \EE \int_{Q^2} |u(t,x) - \tilde u(t,y)| \psi(x) \varrho_\eps(x-y) \dd x \dd y & 
 \leq \EE \int_{Q^2} |u(s,x) - \tilde u(s,y)| \psi(x) \varrho_\eps(x-y) \dd x \dd y \\
 & \quad + \EE \int_s^t \int_Q |A(u(\tau,x)) - A(\tilde u(\tau,x))| \Delta\psi(x)\dd x \dd \tau \\
 & \quad +  C \, \mathcal{G}_\alpha(\delta,\eps,0) \,
 \EE\left(1+\|u\|_{L^{m+1}_{t,x}}^{m+1}+\|\tilde u\|_{L^{m+1}_{t,x}}^{m+1}\right) \\
 & \quad + C \eps \EE \|\nabla \tA(\tu)\|_{L^1_{t,x}}
\end{align*}
for almost every $s\leq t\leq T$. Notice that by \eqref{eq:ic-is-Lebesgue}, $\tau=0$ is a right Lebesgue point of \eqref{eq:leb_points}, hence the last inequality 
holds also for $s=0$. As in \cite[proof of Theorem 4.1]{DGG19}, we pass to the limit $\eps,\delta\to 0$ simultaneously by choosing $\delta$ depending on $\eps$. 
More specifically, we choose $\nu \in((m\wedge 2)^{-1},\bar \kappa)$ and $\al< 1 \wedge \frac{m}{2}$ such that $(2\al)(2\nu)>2$ and set $\delta = \eps^{2\nu}$. Letting $\eps\to 0$ 
proves the desired estimate.
\end{proof}

In the next corollary we replace $\psi$ in Lemma \ref{lem:final_psi}-\ref{it:final_psi_2} by $w$ as in \eqref{eq:w} which implies an estimate in $L^1_{w;x}$.

\begin{corollary} \label{cor:contraction} Under the assumptions of Lemma \ref{lem:final_psi}-\ref{it:final_psi_2} we have that for almost every $s< t\leq T$
\begin{equs}
 & \EE\|u(t) - \tilde u(t)\|_{L^1_{w;x}} \leq \EE\|u(s) - \tilde u(s)\|_{L^1_{w;x}} 
 - \EE\int_s^t  \|A(u(\tau)) - A(\tu(\tau))\|_{L^1_x} \dd \tau. \label{eq:contraction_A}
\end{equs}
In addition, the following estimate holds, 
\begin{equs}
 \supess_{t\in[0,T]} \EE \|u(t) - \tu(t)\|_{L^1_{w;x}} \leq \EE \|\xi - \txi\|_{L^1_{w;x}}. \label{eq:contraction}
\end{equs}
\end{corollary}

\begin{proof} We choose a sequence of non-negative functions $\psi_n \in C^\infty_c(Q)$ such that $\psi_n \to w$ in $H^1_{0,x}$. Using that $|A(u)-A(\tu)| \in L^{2}_{\omega, t}H^1_{0,x}$ and that $w$ solves  \eqref{eq:w}, we get by virtue of Lemma
\ref{lem:final_psi}-\ref{it:final_psi_2} that for almost every $s <t\leq T$ (including $s=0$)
\begin{equs}
 \EE \int_Q |u(t,x) - \tilde u(t,x)| w(x) \dd x \leq  \int_Q |u(s,x) - \tilde u(s,x)| w(x) \dd x 
 - \EE\int_s^t \int_Q |A(u(\tau,x)) - A(\tu(\tau,x))| \dd x \dd \tau. 
\end{equs}
which proves the desired estimate.
\end{proof}

\section{Proofs of well-posedenss} \label{s:well-posedness_proof}

\subsection{Proof of Theorems \ref{thm:uniq_exist} and \ref{thm:stability}} \label{s:proof_uniq_exist_stab}

%
%

In this section we first prove Theorem \ref{thm:uniq_exist} on the existence and uniqueness of solutions to $\mathcal{E}(A,\sigma,\xi)$. From Corollary \ref{cor:contraction} (see \eqref{eq:contraction}) it follows that each two entropy 
solutions of $\mathcal{E}(A,\sigma,\xi)$ coincide, provided that one of them satisfies the $(\star)$-property (see Definition \ref{def:star_prop}). Hence, in order to conclude the existence and 
uniqueness of entropy solutions, it suffices to show the existence of an entropy solution satisfying the $(\star)$-property. To do so, we use a vanishing viscosity approximation.
The (probabilistically) strong existence of solutions for the approximating equations is quite standard by now. It relies on a technique from \cite{GK96}, where a characterisation of the 
convergence in probability is used to show that weak existence combined with strong uniqueness implies strong existence. This has been used in the past in the context of SPDEs 
(see \cite{Hof13,GH18} and the references therein). Proofs  are included in Appendix \hyperlink{ap:A}{A} for the convenience of the reader. 

For the proof of the following proposition we refer the reader to \cite[Proposition 5.1]{DGG19}.

\begin{proposition} \label{prop:Phi-n} Let $A$ satisfy Assumption \ref{as:A}-\ref{as:A first} with a constant $K\geq 1$.
Then, for every $n\geq 1$ there exists an increasing function $A_n\in C^\infty(\bR)$ with bounded derivatives, satisfying 
Assumption \ref{as:A}-\ref{as:A first} with constant $3K$, such that $\fra_n(r)\geq \frac{2}{n}$, and 
\begin{equs} \label{eq:approx A}
\sup_{|r|\leq n}|\fra(r)-\fra_n(r)|\leq \frac{4}{n}.
\end{equs}
\end{proposition}

Let $A_n$ be as above and set 
\begin{equation} \label{def:xi-n}
\xi_n:=(-n) \vee(\xi \wedge n), \, \sigma_n:= \rho_{\frac{1}{n}}^{\otimes(d+1)} * \sigma(\cdot, -n\vee(\cdot \wedge n)).
\end{equation}

\begin{definition} \label{def:L^2_sol} An $L^2_x$-solution of the Dirichlet problem $\mathcal{E}(A_n, \sigma_n, \xi_n)$ is a  continuous $L^2_x$-valued
process $u_n$, such that $u_n, A_n(u_n) \in L^2_{\omega,t} H^1_{0;x}$, and the equality
\begin{equs}
(u_n(t,\cdot),\phi)=(\xi_n,\phi) -\int_0^t(\nabla A_n(u_n(s,\cdot)),\nabla \phi)\,ds
+\int_0^t\left(\sigma^k_n(\cdot,u_n(s,\cdot)), \phi\right) \dd \beta^k(s)
\end{equs}
holds for every $\phi\in C^\infty_c(Q)$, $\PP$-almost surely, for every $t\in[0,T]$.
\end{definition}

If $u_n$ is an $L^2_x$-solution of $\mathcal{E}(A_n, \sigma_n, \xi_n)$, then by standard arguments (see also 
\cite[page 24]{DGG19}) one obtains
\begin{equs}         
\label{eq:change1}     
\E \sup_{t \leq T} \| u_n\|_{L^2_x}^p + \E \|\nabla [\fra_n](u_n) \|_{L^2_{t,x}}^p &\le_p  1+ \E \|\xi_n\|_{L^2_x}^p,
\\         
\label{eq:change2}           
\E \sup_{t \leq T} \|u_n\|_{L^{m+1}_x}^{m+1}+ \E \| \nabla A_n(u_n)\|_{L^2_{t,x}}^2 &\le 1+\E \|\xi_n\|_{L^{m+1}_x}^{m+1},
\end{equs}
where the implicit constants do not depend on $n$. Notice that $|\xi_n|$ is bounded by $n$, which implies that the right
hand side of the last inequalities is finite. Moreover, by the construction of $\xi_n$ one concludes that for every $p\geq 2$
\begin{equs}
\label{eq:uniform}
\E \sup_{t \leq T} \| u_n\|_{L^2_x}^p + \E \|\nabla [\fra_n](u_n) \|_{L^2_{t,x}}^p & \le_p 1+ \E \|\xi\|_{L^2_x}^p,
\end{equs}
and
\begin{equs}
\label{eq:uniform-m+1}
\E \sup_{t \leq T} \|u_n\|_{L^{m+1}_x}^{m+1}+ \E \| \nabla A_n(u_n)\|_{L^2_{t,x}}^2 &\le 1+\E \|\xi\|_{L^{m+1}_x}^{m+1}.
\end{equs}
Finally, since $\fra_n\geq \frac{2}{n}>0$, we have $|\nabla u_n|\leq C_n|\nabla[\fra_n](u_n) |$, and so
by \eqref{eq:uniform}, we have the ($n$-dependent) bound
\begin{equs}\label{eq:gradient}
\E\|\nabla u_n\|_{L^2_{t,x}}^p<\infty.
\end{equs}

The proofs of the next two lemmata are the same as the ones of  \cite[Lemma 5.2]{DGG19} and \cite[Corollary 3.4]{DGG19} and are therefore omitted. 

\begin{lemma} \label{lem:strong entropy}
For $n \geq 1$, let $u_n$ be an $L^2_x$-solution of $\mathcal{E}(A_n, \sigma_n, \xi_n)$. Then, $u_n$ satisfies the $(\star)$-property \eqref{eq:star_prop}
with coefficient $\sigma_n$ and $C\equiv C(\mathcal{S},n)>0$. If in addition $\|\xi\|_{L^2_x}$ has moments of order 4, then the constant $C$ is independent of $n$. 
\end{lemma}

\begin{lemma} \label{lem:star_lim} 
Let $u_n$ be a sequence bounded in $L^{m+1}_{\omega,t} L^{m+1}_x$, satisfying the $(\star)$-property \eqref{eq:star_prop} with coefficient $\sigma_n$, uniformly in $n$. 
Suppose that $u_n$ converges for almost every $(\omega,t,x)$ to a function $u$ and $\lim_{n \to \infty}d(\sigma_n, \sigma)=0$. Then $u$ has the $(\star)$-property with
coefficient $\sigma$.
\end{lemma}

The proof of the next proposition is given in Appendix \hyperlink{ap:A}{A}. 

\begin{proposition} \label{prop:viscoous-well-posedenss}
Suppose Assumptions \ref{as:A} and \ref{as:noise} hold. Then, for every $n \geq 1$, $\mathcal{E}(A_n, \sigma_n,  \xi_n)$ 
has a unique $L^2_x$-solution $u_n$.
\end{proposition}

We are now ready to prove Theorem \ref{thm:uniq_exist}.

\begin{proof}[Proof of Theorem \ref{thm:uniq_exist}] 
Step 1: We first assume that $\E \|\xi\|_{L^2_x}^4< \infty$. For $n\geq 1$, let $u_n$ be the unique $L^2_x$-solution of $\mathcal{E}(A_n, \sigma_n, \xi_n)$ by Proposition 
\ref{prop:viscoous-well-posedenss}. We will show that $(u_n)_{n\geq 1}$ is a Cauchy sequence in $L^1_{\omega,t}L^1_{w;x}$.
Let $N\geq 1$ be arbitrary. As in the conclusion of the proof of Lemma \ref{lem:final_psi}-\ref{it:final_psi_2}, we choose $\nu$ such that $\nu \in ((m\wedge2)^{-1},\bar\kappa)$ 
and then we choose $\alpha<1\wedge\frac{m}{2}$ such that $-2+(2\alpha)(2\nu)>0$, so that $\mathcal{G}_\alpha(\eps^{2\nu},\eps,0) \to 0 $ as
$\eps \to 0$. Then, let $\psi_l \in C^\infty_c(Q)$ be a sequence of non-negative functions such that $\| \psi_l -w\|_{H^1_{0;x}} \to 0$ as $l \to \infty$. We apply Lemma 
\ref{lem:final_psi}-\ref{it:final_psi_1} to $u_n$ and $u_{n'}$, for arbitrary $n\leq n'$, setting $\delta=\eps^{2\nu}$, and $\lambda=\frac{8}{n}$. By 
\eqref{eq:approx A} we have that $R_\lambda\geq n$ (see the statement of Lemma \ref{lem:final_psi}-\ref{it:final_psi_1} for the definition of $R_\lambda$).
Recalling the uniform estimates \eqref{eq:uniform}-\eqref{eq:uniform-m+1} and the triangle inequality
\begin{equs}
\E\|\xi_{n'}(\cdot)-\xi_{n'}(\cdot+h)\|_{L^1_x}\leq
\E\|\xi(\cdot)-\xi(\cdot+h)\|_{L^1_x}+2\E\|\xi-\xi_{n'}\|_{L^1_x},
\end{equs}
where for convenience we have extended $\xi_n$ and $\xi$ on $\RR^d$ by setting them equal to $0$ in $Q^c$, we have
\begin{equs}
\E \| (u_n-u_{n'}) \psi_l\|_{L^1_{t,x}}
& \leq C_l \left(\E\|\xi-\xi_{n'}\|_{L^1_x}+ \E\|\xi-\xi_{n}\|_{L^1_x}\right)
\\
& \quad + \EE \int_0^T \int_0^t \int_Q |A_n(u_n(\tau))- A_{n'}(u_{n'}(\tau))| \Delta\psi_l \dd x \dd \tau \dd t 
\\
&\quad + C_l \eps^{-2}\E\left(\|\mathbf{1}_{|u_n|\geq n}(1+|u_n|)\|_{L^m_{t,x}}^m+
\|\mathbf{1}_{|u_{n'}|\geq n}(1+|u_{n'}|)\|_{L^m_{t,x}}^m\right)
\\
& \quad + C_l \eps^{-2 \nu} d(\sigma_n,\sigma_{n'}) + C_l \eps^{-2}n^{-2} + C_l \eps^{-1}n^{-1} +M_l(\eps). \label{eq:idk-how-to-call-it}
\end{equs}
Here the constants $C_l$ are independent of $\eps,n,n'$ and $M_l(\eps) = \tilde M_l(\eps^{2\nu},\eps)$ for
\begin{equs}
 \tilde M_l(\delta,\eps) = C_l \left(\delta^{2\kappa} + \delta^{-1} \eps^{2\bar\kappa} + \delta \eps^{-1} + \delta^{2\alpha}\eps^{-2} + \eps + 
 \sup_{|h|\leq \eps} \EE\|\xi(\cdot) - \xi(\cdot +h)\|_{L^1_x}\right).
\end{equs}
In particular $M_l(\eps) \to 0$ as $\eps\to 0$ (for every $l$). Notice that
\begin{equs}
 & \EE \int_0^T \int_0^t \int_Q |A_n(u_n(\tau))- A_{n'}(u_{n'}(\tau))| \Delta\psi_l \dd x \dd \tau \dd t \\
 & \quad \leq \EE \int_0^T \int_0^t \int_Q |A_n(u_n(\tau))- A_{n'}(u_{n'}(\tau))| \Delta w\dd x \dd \tau \dd t \\
 & \quad \quad + \EE \int_0^T \int_0^t \int_Q |A_n(u_n(\tau))- A_{n'}(u_{n'}(\tau))| \Delta(\psi_l- w) \dd x \dd \tau \dd t \\
 & \quad \lesssim \EE\|\nabla |A_n(u_n)- A_{n'}(u_{n'})|\|_{L^2_{t,x}} \|\psi_l-w\|_{H^1_{0;x}} \\
 & \quad \lesssim \EE\left(\|\nabla A_n(u_n)\|_{L^2_{t,x}} + \|\nabla A_{n'}(u_{n'})|\|_{L^2_{t,x}}\right) \|\psi_l-w\|_{H^1_{0;x}}
\end{equs}
where in the second step we use that $\Delta w= -1$, integration by parts and the Cauchy--Schwarz inequality. This, by virtue of the uniform estimates
\eqref{eq:uniform} and \eqref{eq:uniform-m+1} combined with \eqref{eq:idk-how-to-call-it}, gives
\begin{equs}
\E \|(u_n-u_{n'})w\|_{L^1_{t,x}}
& \leq C \| \psi_l-w\|_{H^1_{0;x}} + C_l(\E\|\xi-\xi_{n'}\|_{L^1_x}+ \E\|\xi-\xi_{n}\|_{L^1_x})
\\
& \quad + C_l \eps^{-2}\E\left(\|\mathbf{1}_{|u_n|\geq n}(1+|u_n|)\|_{L^m_{t,x}}^m +
\|\mathbf{1}_{|u_{n'}|\geq n}(1+|u_{n'}|)\|_{L^m_{t,x}}^m\right) \\
& \quad + C_l \eps^{-2 \nu} d(\sigma_n,\sigma_{n'})+ C_l \eps^{-2}n^{-2} +  C_l \eps^{-1}n^{-1} + M_l(\eps),
\end{equs}
where $C$ does not depend on $l, \eps, n$, nor $n'$. One can now choose first $l$ large enough and then $\eps>0$ small enough so that for all $n,n'$
large   
\begin{equ}
\E \|(u_n-u_{n'})w\|_{L^1_{t,x}} \leq \frac{1}{N}.
\end{equ}
Therefore, $(u_n)_{n\geq 1}$ converges in $L^1_{\omega,t} L^1_{w;x}$ to a limit $u$. Moreover, by passing to a subsequence, we may also assume 
that 
\begin{equs} 
 \lim_{n \to \infty}u_n=u \label{eq:ae-convergence}  
\end{equs}
for almost every  $(\omega, t, x)\in \Omega \times (0,T)  \times Q$. Consequently, by Lemmata \ref{lem:strong entropy} and \ref{lem:star_lim} and \eqref{eq:uniform-m+1},
we have that $u$ has the $(\star)$-property \eqref{eq:star_prop} with coefficient $\sigma$. In addition, it follows by \eqref{eq:uniform-m+1} that for any $q < m+1$, 
\begin{equs} \label{eq:uniform-integrability}
(|u_n(t,x)|^q)_{n=1}^\infty
\end{equs} 
is uniformly integrable on $\Omega \times [0,T] \times Q$.

We now show that $u$ is an entropy solution. From now on, when we refer to the estimate \eqref{eq:uniform}, we only use it with $p=2$. By \eqref{eq:uniform-m+1}, it follows that $u$ satisfies 
Definition \ref{def:entr_sol}-\ref{item:solution-in-spaces}. Let $f \in C_b(\bR)$ and $\eta$ be as in Definition \ref{def:entr_sol}. For every $n\geq 1$, we clearly have $[\fra_n f](u_n) \in L^2_{\omega,t} H^1_{0,x}$ and 
$\D_{x_i} [\fra_n f](u_n)= f(u_n) \D_{x_i} [\fra_n ](u_n)$. Also, we have $|[\fra_n f](r)|\leq \|f\|_{L^\infty_r} 3K |r|^{\frac{m+1}{2}}$ for every $r \in \bR$, which combined with  \eqref{eq:uniform} and \eqref{eq:uniform-m+1} gives that
\begin{equs}
\sup_n \E \|[\fra_n f](u_n)\|_{L^2_t H^1_{0;x}}^2  < \infty. 
\end{equs} 
Hence, for a subsequence we have the weak convergences $[\fra_n f](u_n)  \wto v_f$, $[\fra_n](u_n) \wto v$ for some $v_f, v \in L^2_{\omega,t}H^1_{0;x}$. 
By \eqref{eq:approx A}, \eqref{eq:ae-convergence} and \eqref{eq:uniform-integrability} it is easy to see that 
$v_f= [\fra f](u) $, $ v= [\fra](u) $. Moreover, for any $\phi \in C^\infty_c([0,T)\times Q)$ and $B \in \mathcal{F}$, we have 
\begin{equs}
\E \mathbf{1}_B \int_0^T \int_Q   \D_{x_i} [\fra f](u) \phi \, \dd x \dd t &= \lim_{n \to \infty} \E \mathbf{1}_B \int_0^T \int_Q \D_{x_i} [\fra_n f](u_n) \phi \dd x \dd t
\\
&= \lim_{n \to \infty} \E \mathbf{1}_B \int_0^T \int_Q f(u_n) \D_{x_i} [\fra_n ](u_n) \phi \dd x \dd t
\\
&= \E \mathbf{1}_B \int_0^T \int_Q f(u) \D_{x_i} [\fra ](u) \phi \dd x \dd t ,
\end{equs}
where for the last equality we have used that $ \D_{x_i}[\fra_n](u_n)  \wto [\fra](u) $ (weakly) and $f(u_n) \to f(u)$ (strongly) in $L^2_{\omega,t} L^2_x$.
Hence, Definition \ref{def:entr_sol}-\ref{item:solution-in-spaces-2} is satisfied. We now show Definition \ref{def:entr_sol}-\ref{item:entropies}. Let $\eta$ and $\phi=\varphi\varrho$ be 
as in Definition \ref{def:entr_sol}-\ref{item:entropies} and let $B \in \mathcal{F}$. By It\^o's formula (see, e.g., \cite{Kry13}) for the function 
\begin{equs}
 u \mapsto \int_Q \eta(u)\varrho \dd x,
\end{equs}
and It\^o's product rule, we have 
\begin{equs}                     
-\E \mathbf{1}_B \int_0^T \int_Q  \eta(u_n)\D_t\phi \dd x \dd t
& =\E \mathbf{1}_B \left[ \int_Q\eta(\xi_n)\phi(0)\dd x+ \int_0^T \int_Q [\eta' \fra_n^2](u_n) \Delta \phi \dd x \dd t  \right.
\\   
& \quad + \int_0^T \int_Q \left( \frac{1}{2}  \phi \eta''(u_n)|\sigma_n(u_n)|_{\ell^2}^2-\phi \eta''(u_n) | \nabla [\fra_n](u_n)|^2 \right) \dd x \dd t
\\ 
& \quad + \left. \int_0^T \int_Q \phi \eta'(u_n)\sigma_n^k(u_n) \dd x  \dd \beta^k(t) \right]. \label{eq:un-entropy-inequality}
\end{equs}
On the basis of \eqref{eq:ae-convergence} and \eqref{eq:uniform-integrability} and the construction of $\xi_n$, $\sigma_n$ and $\mathfrak{a}_n$ it
is easy to see that 
\begin{equs}
\lim_{n \to \infty} \E \mathbf{1}_B \int_Q \eta(\xi_n)\phi(0) \dd x & = \E \mathbf{1}_B \int_Q\eta(\xi)\phi(0) \dd x 
\\
\lim_{n \to \infty} \E \mathbf{1}_B \int_0^T \int_Q \eta(u_n) \D_t\phi \dd x \dd t & = \E \mathbf{1}_B \int_0^T \int_Q \eta(u) \D_t\phi \dd x \dd t
\\
\lim_{n \to \infty}\E \mathbf{1}_B\int_0^T \int_Q  [\eta^{\prime} \fra_n^2](u_n) \Delta \phi \dd x \dd t &  = \E \mathbf{1}_B \int_0^T \int_Q  [\eta^{\prime} \fra^2](u)\Delta \phi\dd x \dd t
\\
\lim_{n \to \infty}\E \mathbf{1}_B \int_0^T \int_Q \phi \eta^{\prime \prime} (u_n)|\sigma_n(u_n)|_{\ell^2}^2 \dd x \dd t & = \E \mathbf{1}_B \int_0^T \int_Q \phi \eta^{\prime \prime} (u)|\sigma(u)|_{\ell^2}^2 \dd x \dd t
\\
\lim_{n \to \infty}\E \mathbf{1}_B\int_0^T \int_Q \phi \eta^{\prime} (u_n)\sigma^k_n(u_n)\dd x \dd\beta^k(t) & = \E \mathbf{1}_B \int_0^T \int_Q \phi \eta^{\prime}(u)\sigma^k(u) \dd x \dd \beta^k(t).
\end{equs}
Let us set $\tilde{f}(r):= \sqrt{\eta''(r)}$. Notice that $\D_{x_i} [\tilde{f} \fra_n](u_n)= \sqrt{\eta''(u_n)} \D_{x_i}[\fra_n] (u_n)$. 
As before we have (after passing to a subsequence if necessary) $\D_{x_i} [\tilde{f} \fra_n](u_n) \wto\D_{x_i} [\tilde{f} \fra](u)$ (weakly) in $L^2_{\omega,t}L^2_x$. In particular, 
this implies that $\D_{x_i} [\tilde{f} \fra_n](u_n) \wto\D_{x_i} [\tilde{f} \fra](u)$ (weakly) in $L^2(\Omega \times(0,T)\times Q;\mathrm{d}\bar\mu)$, where $\dd \bar\mu:= \mathbf{1}_B\phi\dd \bP\otimes \dd x \otimes \dd t $. 
This implies that 
\begin{equs}
\E \mathbf{1}_B\int_0^T \int_Q\phi \eta''(u) | \nabla [\fra](u)|^2\dd x \dd t   \leq \liminf_{n \to \infty} \E \mathbf{1}_B \int_0^T \int_Q \phi \eta''(u_n) | \nabla [\fra_n](u_n)|^2\dd x \dd t .
\end{equs}
Hence, taking $\liminf$ in \eqref{eq:un-entropy-inequality} along an appropriate subsequence, we see that $u$ satisfies Definition \ref{def:entr_sol}-\ref{item:entropies} too.

To summarise, we have shown that if in addition to the assumptions of Theorem \ref{thm:uniq_exist} we have that $\E\| \xi\|_{L^2_x}^4<\infty$, then
there exists an entropy solution to \eqref{eq:spm} which has the $(\star)$-property \eqref{eq:star_prop} with coefficient $\sigma$ 
(therefore, it is also unique by \eqref{eq:contraction} in Corollary \ref{eq:contraction}). In addition, we can pass to the limit in 
\eqref{eq:uniform} and \eqref{eq:uniform-m+1} to obtain that 
\begin{equs} 
\E \sup_{t \leq T} \| u\|_{L^2_x}^2 + \E \|\nabla[\fra](u)\|_{L^2_{t,x}}^2 &\le 1+ \E \|\xi\|_{L^2_x}^2, \label{eq:uniform2_L^2}
\\
\E \sup_{t \leq T} \|u\|_{L^{m+1}_{t,x}}^{m+1}+ \E \| \nabla A(u)\|_{L^2_{t,x}}^2 &\le 1+\E \|\xi\|_{L^{m+1}_x}^{m+1}. 
\label{eq:uniform2_L^{m+1}}
\end{equs}

\smallskip

Step 2: We now remove the extra condition on $\xi$. For $n \geq 1$, let $\xi_n$ be as in \eqref{def:xi-n} and let $u_{(n)}$ be the unique solution of 
$\mathcal{E}(A, \sigma, \xi_n)$. Notice that $u_{(n)}$ has the $(\star)$-property with coefficient $\sigma$. Hence, by Corollary \ref{eq:contraction}
we have that $(u_{(n)})_{n\geq 1}$ is Cauchy in $L^1_{\omega,t}  L^1_{w;x}$ and therefore has a limit $u$. In addition, $u_{(n)}$ satisfy the estimates 
\eqref{eq:uniform2_L^2} and \eqref{eq:uniform2_L^{m+1}} uniformly in $n$. With the arguments provided in Step 1 we can show that $u$ is an
entropy solution. 

\smallskip

Step 3: We finally show \eqref{eq:main_contraction} which also implies uniqueness. Let $\tilde{u}$ be an entropy solution of $\mathcal{E}(A,\sigma, \tilde{\xi})$. By \eqref{eq:contraction} we have that
\begin{equs}
\supess_{t\in[0,T]}\E\|u_{(n)}(t)-\tu(t)\|_{L^1_{w;x}} \leq\E\|\xi_n-\txi\|_{L^1_{w;x}} ,
\end{equs}
for the sequence $u_{(n)}$ as in Step 2. The proof is complete if we let $n \to \infty$.
\end{proof}

\begin{remark} \label{rem:L^2_cont}
Let the assumptions of Theorem \ref{thm:uniq_exist} hold.  If we further assume that $\inf_{r\geq0} \fra(r)=c>0$,  
it is easy to see that in addition to \eqref{eq:uniform2_L^2} and \eqref{eq:uniform2_L^{m+1}} we have 
\begin{equs}
\E \| u\|^2_{L^2_t H^1_{0;x}} \leq C \left( 1+ \E \|\xi\|_{L^2_x}^2\right),
\end{equs}
with $C$ depending on $\mathcal{S}$ and $c$. Furthermore, after a standard approximation argument, it follows from Definition 
\ref{def:entr_sol} that for each $\phi \in H^1_{0;x}$ we have 
\begin{equs}
(u(t), \phi)= (\xi, \phi)-\int_0^t(\nabla A(u(s)),\nabla \phi)\,ds
+\int_0^t\left(\sigma^k(u(s)), \phi\right) \dd \beta^k(s)\
\end{equs}
for almost every $(\omega, t)$. These two facts imply by virtue of \cite[Theorem 3.2]{KR79} that $u$ is a continuous 
$L^2_x$-valued process.
\end{remark} 

We now proceed with the proof of Theorem \ref{thm:stability} which implies the stability of entropy solutions with respect to the initial condition 
$\xi$, the non-linearity $A$ and $\sigma$. The proof is similar to \cite[Proof of Theorem 2.2]{DGG19}. 

\begin{proof}[Proof of Theorem \ref{thm:stability}] Let $\bar \xi_n  = -a_n \vee (a_n \wedge \xi_n)$ where $a_n$ is chosen large enough
such that 
\begin{equs}
 \EE\|\xi_n - \bar \xi_n\|_{L^1_{w;x}} \leq \frac{1}{n}. 
\end{equs}
Let $\bar u_n$ be the solution of $\mathcal{E}(A_n,\sigma_n,\bar\xi_n)$ which by Theorem \ref{thm:uniq_exist} exists, is unique and by Lemma \ref{lem:star_lim}
satisfies the $(\star)$-property \eqref{eq:star_prop}. By \eqref{eq:contraction} we know that 
\begin{equs}
 \EE\|u_n - \bar u_n\|_{L^1_tL^1_{w;x}} \leq \EE \int_0^T \|\xi_n - \bar \xi_n\|_{L^1_{w;x}} \dd t \leq \frac{T}{n}.
\end{equs}
Thus, it is enough to prove that $\bar u_n \to u$ in $L^1_{\omega,t}L^1_{w;x}$. Let $\psi_l\in C^\infty_c(Q)$ be a sequence of
positive functions such that $\|\psi_l-w\|_{H^1_{0;x}}\to 0$ as $l\to\infty$. Then, for arbitrary $l\geq 1$, we have that
\begin{equs}
 \EE\|\bar u_n - u\|_{L^1_tL^1_{w;x}} & \leq \EE \|(\bar u_n - u)\psi_l\|_{L^1_tL^1_x} +
 \EE \|(\bar u_n - u)(w-\psi_l)\|_{L^1_tL^1_x} \\
 & \leq \EE \|(\bar u_n - u)\psi_l\|_{L^1_tL^1_x} + C \|w-\psi_l\|_{H^1_{0;x}}
\end{equs}
for some $C$ independent on $n$ (similarly to \eqref{eq:uniform2_L^2}) and $l$. Using Lemma \ref{lem:final_psi}-\ref{it:final_psi_1} and proceeding similarly to
Step 1 in the proof of Theorem \ref{thm:uniq_exist}, where the specific choice of $\lambda=\frac{8}{n}$ implies that $R_\lambda \geq b_n$ for 
some $b_n\geq 1$ which can be chosen such that $b_n \to \infty$ as $n\to\infty$ (since $A_n \to A$ uniformly on compact sets by assumption), we obtain
that
\begin{equs}
 \EE\|\bar u_n - u\|_{L^1_tL^1_{w;x}} & \leq C \|w-\psi_l\|_{H^1_{0,x}} + C_l \EE\|\bar \xi_n - \xi\|_{L^1_x} \\ 
 & \quad + C_l \eps^{-2}\E\left(\|\mathbf{1}_{|\bar u_n|\geq b_n}(1+|\bar u_n|)\|_{L^m_{t,x}}^m +
 \|\mathbf{1}_{|u|\geq b_n}(1+|u|)\|_{L^m_{t,x}}^m \right) \\
 & \quad + C_l \eps^{-2\nu} d(\sigma_n,\sigma) + C_l \eps^{-2} n^{-2} + C_l \eps^{-1} n^{-1}
 + M_l(\eps),
\end{equs}
for some $M_l(\eps)\to 0$ as $\eps \to 0$ (for fixed $l$) and constants $C_l$ independent of $\eps,n$ and $C$ independent 
of $\eps,l,n$. To conclude, given $N\geq 1$ we first choose $l$ large enough and then $\eps$ small enough so that for all
$n$ large enough
\begin{equs}
 \EE\|\bar u_n - u\|_{L^1_tL^1_{w;x}} \leq \frac{1}{N}
\end{equs}
which completes the proof. 
\end{proof}

\subsection{Proof of Theorem \ref{thm:L^1_cont} and Proposition \ref{prop:def-of-v}} \label{s:L^1_ext_proof}

To prove that entropy solutions belong to $C([0,T];L^1_\omega L^1_{w;x})$ we use the continuity of vanishing viscosity approximations
from Proposition \ref{prop:viscoous-well-posedenss} and the stability Theorem \ref{thm:stability} together with Lemma \ref{lem:almost_final}.

\begin{proof}[Proof of Theorem \ref{thm:L^1_cont}] Let $u_n$ be given by Proposition \ref{prop:viscoous-well-posedenss}. Recall that by Theorem
\ref{thm:stability} we have 
\begin{equation*}
\lim_{n\to 0}\E \|u-u_n\|_{L^1_tL^1_{w;x}}=0,
\end{equation*}
which in particular implies that there exists $\mathcal{T} \subset [0,T]$ with $|\mathcal{T}|=T$ and a (non-relabelled) subsequence such that  
\begin{equs}
 \lim_{n \to \infty} \E\|u(t)-u_n(t)\|_{L^1_{w;x}}=0,
\end{equs}
for every $t \in \mathcal{T}$. We now show that given $\lambda>0$, there exists $h$ such that 
\begin{equs} 
 \E \| u_n(t)-u_n(t')\|_{L^1_{w;x}} \leq \lambda, \label{eq:uniform-continuity}
\end{equs}
for every $n \geq 1$ and $t, t' \in \mathcal{T}$ with $|t-t'|\leq h$. By \eqref{eq:uniform-m+1} we have 
\begin{equs} \label{eq:with-compact-support}
 \E \|u_n(t)-u_n(t')\|_{L^1_{w;x}} \le_{\xi} \|w-\psi_l\|_{L^2_x}+\E \int_Q |u_n(t)-u_n(t')| \psi_l \dd x,
\end{equs}
where $\psi_l \in C^\infty_c(Q)$, $\psi_l \geq 0 $ and $\|w-\psi_l\|_{H^1_{0;x}} \to 0$. We also have 
\begin{align}
 \E \int_Q |u_n(t,x)-u_n(t',x)| \psi_l(x) \, dx & \leq \E \int_Q \int_Q |u_n(t,x)-u_n(t',y)| \psi_l(x)\varrho_\eps(x-y) \dd x \dd y 
 && (=:I_1) \nonumber\\                                   
 & \quad +\E \int_Q \int_Q |u_n(t',x)-u_n(t',y)| \psi_l(x)\varrho_\eps(x-y) \dd x \dd y. && (=:I_2) \label{eq:decomposition-I_1-I_2}
\end{align}
%

By Lemma \ref{lem:almost_final} and using that $u_n$ is continuous in $t$ with values in $L^2_x$ (see Remark \ref{rem:L^2_cont}) for every $n\geq 1$ (hence every $t\in[0,T]$ is 
a Lebesgue point of \eqref{eq:leb_points}), we obtain 
\begin{equs}
 I_2 & = \E \int_Q \int_Q |u_n({t^\prime},x)-u_n(t^\prime,y)| \psi_l(x)\varrho_\eps(x-y) \dd x \dd y \\
 & \leq 
 \ \E  \int_Q \int_Q |\xi_n(x)-\xi_n(y)| \psi_l(x)\varrho_\eps(x-y) \dd x \dd y 
 \\
 & \quad + \EE \int_0^{t^\prime} \int_{Q^2}|A_n(u_n(s,x)) - A_n(u_n(s,y))| \Delta \psi_l(x) \varrho_\eps(x-y) \dd x \dd y \dd t  \\
 & \quad +  C \mathcal{G}_\alpha (\delta, \eps,0)
 \EE\left(1+\|u_n\|_{L^{m+1}_{t,x}}^{m+1}\right)
 \\
 & \leq \EE  \int_Q \int_Q |\xi(x)-\xi(y)| \psi_l(x)\varrho_\eps(x-y) \dd x \dd y 
 \\
 & \quad + C \eps \EE \|\nabla A_n(u_n)\|_{L^1_{t,x}} \| \Delta \psi_l\|_{L^\infty_x}  +  C_l \mathcal{G}_\alpha (\delta, \eps,0)
 \EE\left(1+\|u_n\|_{L^{m+1}_{t,x}}^{m+1}\right)
\end{equs}
for some constant $C_l>0$ which depends on $l$ but not on $\delta$ and $\eps$. Using \eqref{eq:uniform-m+1} we conclude that 
\begin{equs} 
 I_2 \le_{\xi}   \ \E  \int_Q \int_Q |\xi(x)-\xi(y)| \psi_l(x)\varrho_\eps(x-y) \dd x \dd y
 +\eps  \| \Delta \psi_l\|_{L^\infty_x} + M_l(\eps), \label{eq:est-I_2}
\end{equs}
where $M_l(\eps) \to 0 $ as $\eps \to 0$ (for fixed $l$, by choosing $\delta$ and $\alpha$ as in the proof of Lemma \ref{lem:final_psi}-\ref{it:final_psi_2}).
We also have that
\begin{equs}
 I_1 \le  \delta \|\psi_l\|_{L^\infty_x} + \E \int_Q \int_Q \eta_\delta(u_n(t,x)-u_n(t',y)) \psi_l(x)\varrho_\eps(x-y) \dd x \dd y.
\end{equs}
By It\^o's formula we conclude that 
\begin{equs}
\E & \int_Q \int_Q \eta_\delta (u_n(t,x)-u_n(t',y)) \psi_l(x)\varrho_\eps(x-y) \dd x \dd y
\\
& \leq \E \int_Q \int_Q \eta_\delta(u_n(t',x)-u_n(t',y)) \psi_l(x)\varrho_\eps(x-y) \dd x \dd y
\\
& \quad + \E \int_{t'}^t \int_Q \int_Q \eta_\delta '(u_n(s,x)-u_n(t',y)) \nabla A_n(u_n(s,x)) \nabla(\psi_l(x)\varrho_\eps(x-y) ) \dd x \dd y \dd s 
\\
& \quad + \E \int_{t'}^t \int_Q \int_Q \frac{1}{2}\eta_\delta ''(u_n(s,x)-u_n(t',y)) |\sigma_n(x,u_n(s,x))|_{\ell^2}^2 \psi_l(x)\varrho_\eps(x-y) \dd x \dd y \dd s,
\end{equs}
where we have used that 
\begin{equs}
 -\E \int_{t'}^t \int_Q \int_Q \eta_\delta''(u_n(s,x)-u_n(t',y)) \nabla u_n(s,x) \nabla A_n(u_n(s,x)) \psi_l(x)\varrho_\eps(x-y) \dd x \dd y \dd s
 \leq 0.
\end{equs}
From this, using that $|\sigma_n(x,u_n(s,x))|_{\ell^2}^2\lesssim 1 + |u_n(s,x)|^2$, it follows that 
\begin{equs}
I_1 & \le \delta   \|\psi_l\|_{L^\infty_x}+ I_2 + \|\psi_l\|_{W^{1,\infty}_x} \int_{t'}^t \eps^{-1} \E \| \nabla A_n(u_n(s))\|_{L^2_x}  +\delta^{-1} \left(1+\E \|u_n(s)\|_{L^2_x}^2\right) \dd s
 \\
& \le \  \delta   \|\psi_l\|_{L^\infty_x}+ I_2+ \eps^{-1}|t'-t|^{\frac{1}{2}} \E \| \nabla A_n(u_n)\|_{L^2_{t,x}}+ \delta^{-1}|t'-t|^{\frac{m-1}{m+1}}\left(1+\E \|u_n(s)\|_{L^{m+1}_{t,x}}^2\right)
\\
& \le_{\xi} \delta   \|\psi_l\|_{L^\infty_x}+ I_2+ \eps^{-1}|t'-t|^{\frac{1}{2}} + \delta^{-1}|t'-t|^{\frac{m-1}{m+1}}. \label{eq:est-I_1}                        
\end{equs}
Consequently, by \eqref{eq:with-compact-support}-\eqref{eq:est-I_1}, we obtain
\begin{equs} 
\E \|u_n(t)-u_n(t')\|_{L^1_{w;x}} & \le_{\xi} \|w-\psi_l\|_{L^2_x} + \E  \int_Q \int_Q |\xi(x)-\xi(y)| \psi_l(x)\varrho_\eps(x-y) \dd x \dd y
+\eps  \| \Delta \psi_l\|_{L^\infty_x}
\\
& \quad  + M_l(\eps) + \delta   \|\psi_l\|_{L^\infty_x}+ \eps^{-1}|t'-t|^{\frac{1}{2}} + \delta^{-1}|t'-t|^{\frac{m-1}{m+1}}.   
\end{equs}
Then, given $\lambda>0$, we choose $l$  large enough and $\eps, \delta >0$ small enough so that 
\begin{equs}
 \|w-\psi_l\|_{L^2_x} + \E  \int_Q \int_Q |\xi(x)-\xi(y)| \psi_l(x)\varrho_\eps(x-y) \dd x \dd y
 +\eps  \| \Delta \psi_l\|_{L^\infty_x}+ M_l(\eps) + \delta   \|\psi_l\|_{L^\infty_x} \leq \lambda/2.
\end{equs}
Then it is clear that for every $t', t \in \mathcal{T}$ such that $|t'-t |$ is sufficiently small we have 
\begin{equs}
 \E \| u_n(t)-u_n(t')\|_{L^1_{w;x}} \leq \lambda,
\end{equs}
which after passing to the limit $n \to \infty$ gives 
\begin{equs}
 \E \| u(t)-u(t')\|_{L^1_{w;x}} \leq \lambda.
\end{equs}
Consequently $u : \mathcal{T} \to L^1_\omega L^1_{w;x}$ is uniformly continuous, hence it has a unique continuous
extension on $[0,T]$.
\end{proof}

\begin{proof}[Proof of Proposition \ref{prop:def-of-v}] Let $(\xi_n)_{n\geq 1}$ be a sequence in $L^{m+1}_\omega L^{m+1}_x$ such that $\xi_n\to\xi$ in $L^1_\omega L^1_{w;x}$.
By Theorem \ref{thm:uniq_exist} for every $m>n\geq 1$ we have the estimate
\begin{equs}
 \sup_{t\in [0,T]} \EE\|u(t;\xi_m) - u(t;\xi_n)\|_{L^1_{w;x}} \leq \EE\|\xi_m - \xi_n\|_{L^1_{w;x}} 
\end{equs}
where we have replaced $\supess_{t\in[0,T]}$ by $\sup_{t\in[0,T]}$ in \eqref{eq:main_contraction} since, by Theorem \ref{thm:L^1_cont}, we know that 
$u(\cdot;\xi_n)\in C([0,T];L^1_\omega L^1_{w;x})$, for every $n\geq 1$. Thus, the sequence $\{u(\cdot;\xi_n)\}_{n\geq 1}$ is Cauchy in 
$C([0,T];L^1_\omega L^1_{w;x})$, which implies that the limit $v(\cdot;\xi):=\lim_{n\to\infty} u(\cdot;\xi_n)$ exists in $C([0,T];L^1_\omega L^1_{w;x})$.
Moreover, if $(\txi_n)_{n\geq 1}$ is another sequence in $L^{m+1}_\omega L^{m+1}_x$ which converges to $\xi$ in $L^1_\omega L^1_{w;x}$, using again Theorem \ref{thm:uniq_exist},
we get that for every $n\geq 1$,
\begin{equs}
 \sup_{t\in [0,T]} \EE\|u(t;\xi_n) - u(t;\txi_n)\|_{L^1_{w;x}} \leq \EE\|\xi_n - \txi_n\|_{L^1_{w;x}}
\end{equs}
which shows that $v(\cdot;\xi)$ is independent of $(\xi_n)_{n\geq 1}$. It is easy to see that for $\xi\in L^{m+1}_\omega L^{m+1}_x$, $v(\cdot;\xi)$ coincides 
with $u(\cdot;\xi)$ in $C([0,T];L^1_\omega L^1_{w;x})$. Thus $v$ is the unique continuous (with respect to the variable $\xi$) extension of $u$ on $C([0,T];L^1_\omega L^1_{w;x})$. To ease the notation we
identify $u$ with $v$. Finally, \eqref{eq:v_contr} follows easily by construction.  
\end{proof}

\section{Proofs of ergodicity} \label{s:ergod_proof}

\subsection{Proof of Proposition \ref{prop:L^m+1_bddness}}

This is a simple application of Fatou's lemma combined with Lemma \ref{lem:supess_L^m+1} and Theorem \ref{thm:L^1_cont}. 

\begin{proof}[Proof of Proposition \ref{prop:L^m+1_bddness}] Let $t\in[0,\infty)$. By Theorem \ref{thm:L^1_cont} we know that $u(\cdot;\xi)\in C([0,\infty);L^1_\omega L^1_{w;x})$,
hence there exists a sequence $t_n\to t$ such that $u(t_n;\xi) \to u(t;\xi)$ for almost every $(\omega,x)$. By Lemma 
\ref{lem:supess_L^m+1} we can assume that
\begin{equs}
 \sup_{n\geq 1} (t_n\wedge 1)^{\frac{m+1}{m-1}} \EE\|u(t_n;\xi)\|_{L^{m+1}_x}^{m+1} \leq C
\end{equs}
for some $C>0$ depending only on $d$, $K$, $m$, and $|Q|$, but not on $t$ and $\xi$. Then, by Fatou's lemma we have that
\begin{equs}
 (t\wedge 1)^{\frac{m+1}{m-1}}\EE\|u(t;\xi)\|_{L^{m+1}_x}^{m+1} & \leq \liminf_{n\to\infty} (t_n\wedge 1)^{\frac{m+1}{m-1}} \EE\|u(t_n;\xi)\|_{L^{m+1}_x}^{m+1} 
 \leq \sup_{n\geq 1} (t_n\wedge 1)^{\frac{m+1}{m-1}} \EE\|u(t_n;\xi)\|_{L^{m+1}_x}^{m+1} \leq C
\end{equs}
which completes the proof since $t$ is arbitrary.
\end{proof}

\subsection{Proof of Theorem \ref{thm:pol_contr}} \label{s:pol_contr}

In this section we prove Theorem \ref{thm:pol_contr}. The proof is based on \eqref{eq:contraction_A} in Corollary \ref{cor:contraction}, which however
we have only shown for entropy solutions satisfying the $(\star)$-property \eqref{eq:star_prop} and for almost every $s<t$. For this reason we first use the approximations from Proposition
\ref{prop:viscoous-well-posedenss} which satisfy the $(\star)$-property and pass to the limit in \eqref{eq:contraction_A}. We then use the continuity of solutions in
$L^1_\omega L^1_{w;x}$ given by Theorem \ref{thm:L^1_cont} to derive the estimate for every $s<t$.   

\begin{proof}[Proof of Theorem \ref{thm:pol_contr}.] We first assume that $\xi, \txi\in L^{m+1}_\omega L^{m+1}_x$. By Theorem \ref{thm:stability} we can approximate $u(\cdot;\xi)$, $u(\cdot;\txi)$
in $L^1_{\omega,t}L^1_{w;x}$ by $u_n$, $\tu_n$ as in Proposition \ref{prop:viscoous-well-posedenss} satisfying the $(\star)$-property. By \eqref{eq:contraction_A} in Corollary 
\ref{cor:contraction} and by passing to the limit $n\to\infty$ (upon a subsequence) we obtain that for almost every $0\leq s<t\leq T$ (including $s=0$)
\begin{equs}
 & \EE\|u(t;\xi) - u(t;\txi)\|_{L^1_{w;x}} - \EE\|u(s;\xi) - u(s;\txi)\|_{L^1_{w;x}} \\ 
 & \quad \leq - \EE \int_s^t \||u(\tau;\xi)|^{m-1} u(\tau;\xi) - |u(\tau;\txi)|^{m-1} u(\tau;\txi)\|_{L^1_x} \dd \tau. 
\end{equs}
Since by Theorem \ref{thm:L^1_cont} we know that $u(\cdot;\xi), u(\cdot;\txi)\in C([0,T];L^1_\omega L^1_{w;x})$, the same estimate holds for every $0\leq s< t\leq T$.
Combining this estimate with Lemma \ref{lem:lwr_bd}, there exists $C>0$ depending only on $m$ such that
\begin{equs}
 \EE \|u(t;\xi) - \tilde u(t;\txi)\|_{L^1_{w;x}} - \EE \|u(s;\xi) - u(s;\txi)\|_{L^1_{w;x}}
 \leq - C \int_s^t \EE\|u(\tau;\xi) - u(\tau;\txi)\|_{L^m_x}^m \dd \tau
\end{equs}
and if we furthermore notice that 
\begin{align*}
 \left(\EE\|u(\tau;\xi) - u(\tau;\txi)\|_{L^1_{w;x}}\right)^m \leq \EE\|u(\tau;\xi) - u(\tau;\txi)\|_{L^m_x}^m \|w\|_{L^{m_*}_x}^m
\end{align*}
we finally obtain the integral inequality 
\begin{align*}
 \EE \|u(t;\xi) - u(t;\txi)\|_{L^1_{w;x}} - \EE \|u(s;\xi) - u(s;\txi)\|_{L^1_{w;x}} \leq 
 - C \|w\|_{L^{m_*}_x}^{-m} \int_s^t \left(\EE\|u(\tau;\xi) - u(\tau;\txi)\|_{L^1_{w;x}}\right)^m \dd \tau.
\end{align*}
Let $f(t) = \EE \|u(t;\xi) - u(t;\txi)\|_{L^1_{w;x}}$. Then $f$ satisfies
the integral inequality
\begin{align*}
 f(t) - f(s) \leq - C \|w\|_{L^{m_*}_x}^{-m} \int_s^t f(\tau)^m \dd \tau,
\end{align*}
for every $s\leq t$ and it is continuous since $u,\tu\in C([0,T];L^1_\omega L^1_{w;x})$. Hence, by Lemma \ref{lem:int_ineq} 
we obtain that $f(t) \leq h(t)$ where $h$ solves 
\begin{align*}
 \begin{cases}
  & h'(t) = - C \|w\|_{L^{m_*}_x}^{-m} h(t)^m \\
  & h(0) = f(0).
 \end{cases}
\end{align*}
A simple computation shows that
\begin{align*}
 h(t) = \left( \frac{1}{h(0)^{-(m-1)} +  C \|w\|_{L^{m_*}_x}^{-m} (m-1) t}\right)^{\frac{1}{m-1}}
\end{align*}
which in turn implies that
\begin{align*}
 \EE \|u(t;\xi) - u(t;\txi)\|_{L^1_{w;x}} \leq C \|w\|_{L^{m_*}_x}^{\frac{m}{m-1}} t^{-\frac{1}{m-1}}.
\end{align*}

If $\xi,\txi\in L^1_\omega L^1_{w;x}$, we approximate with sequences $(\xi_n)_{n\geq 1}, (\txi_n)_{n\geq 1}$ in $L^{m+1}_\omega L^{m+1}_x$ and use 
\eqref{eq:v_contr} to pass to the limit. 
\end{proof}

\subsection{Proofs of Proposition \ref{prop:Markov} and Theorem \ref{thm:pol_mix}} \label{s:pol_mix_proof}

To prove Proposition \ref{prop:Markov}, we first prove the flow property for entropy solutions (see Corollary \ref{cor:restart}). This uses the continuity of solutions
in $L^1_\omega L^1_{w;x}$ given by Theorem \ref{thm:L^1_cont}, in combination with an approximation argument in $L^{m+1}_\omega L^{m+1}_x$ in the spirit of \cite[Theorem 9.14]{DPZ14}, 
which strongly relies on the weighted $L^1$-contraction estimate \eqref{eq:v_contr}. For technical reasons in this section we extend the time horizon to $-\infty$.      

We need the following notation. For $\xi \in L^{m+1}_x$ and $s>-\infty$ we denote by $u_s(\cdot;\xi)$ the entropy solution to
\begin{equation} \label{eq:shifted_spm}
 \begin{cases}
  & \partial_t u_s(t,x;\xi) = \Delta\left(|u_s(t;\xi)|^{m-1}u_s(t;\xi)\right)(t,x) + \sigma^k(x,u_s(t,x;\xi)) \dot{\beta}^k(t) \\
  & u_s(s,x;\xi) = \xi \\
  & u_s|_{\partial Q} = 0
 \end{cases}
\end{equation}
for $t\geq s$, where we have extended $\beta^k(t)$ for $t<0$ by gluing at $t=0$ an independent Brownian motion evolving backwards in time. The existence and uniqueness of entropy solutions to this equation for $s=0$ is given by Theorem  \ref{thm:uniq_exist}. 
The case $s\neq 0$ follows analogously. To be consistent with the previous sections, we simply write $u(\cdot;\xi)$ to denote $u_0(\cdot;\xi)$. 

We have the following useful relation. 

\begin{corollary} \label{cor:restart} For every $\xi\in L^{m+1}_x$ and $-\infty < s_1\leq s_2 \leq t\leq T$ we have that $u_{s_1}(t;\xi) = u_{s_2}(t;u_{s_1}(s_2;\xi))$
in $L^1_\omega L^1_{w;x}$. 
\end{corollary}

\begin{proof} For simplicity we prove the statement for $(s_1,s_2) = (0,s)$. If $u$ is an entropy solution on $[0,T]$
and $\varphi\in C^\infty_c([s,T))$, which can be extended continuously on $[0,s]$ (taking the constant value $\varphi(s)$),
we can choose $\phi =  \chi_k \varphi \varrho$ in Definition \ref{def:entr_sol} with 
\begin{equation*}
  \chi_\kappa(t) = 
  \begin{cases}
   1, & t\in[s,T) \\
   \kappa\left(t- \left(s - \frac{1}{\kappa}\right)\right), & t\in \left[s-\frac{1}{\kappa}, s\right] \\ 
   0, & t\in \left[0, s-\frac{1}{\kappa}\right]
 \end{cases}
\end{equation*}
to obtain that
\begin{equs}
 - \int_0^T \int_Q \eta(u) \chi_\kappa \partial_t \varphi(t) \varrho \dd x \dd t & \leq \kappa\int_{s-\frac{1}{\kappa}}^s \int_Q \eta(u) \varphi(s) \varrho \dd x \dd t 
 + \int_0^T \int_Q \left[\eta'\fra^2\right](u) \chi_\kappa  \varphi \Delta\rho \dd x \dd t \\
 & \quad + \int_0^T \int_Q \left(\frac{1}{2} \chi_\kappa \varphi \varrho \eta''(u) |\sigma(u)|_{\ell^2}^2 - \chi_\kappa \varphi \varrho \eta''(u) |\nabla [\fra](u)|^2\right) \dd x \dd t \\
 & \quad + \int_0^T \int_Q \chi_\kappa \varphi \varrho \eta'(u) \sigma^k(u) \dd x \dd \beta^k(t) \label{eq:entr_sol_s}
\end{equs}
One can prove that $u(s)\in L^{m+1}_\omega L^{m+1}_x$. Indeed, we can approximate $u$ by $u_n$
as in Proposition \ref{prop:viscoous-well-posedenss} in $L^1_{\omega,t} L^1_{w;x}$, and hence for almost every $(\omega,t,x)$ up to a subsequence, and use 
\eqref{eq:uniform-m+1} and Fatou's lemma to obtain that 
\begin{equs}
 \supess_{t\in[0,T]} \EE\|u(t)\|_{L^{m+1}_x}^{m+1} \lesssim 1+ \EE\|\xi\|_{L^{m+1}_x}^{m+1}. 
\end{equs}
This estimate together with the fact that $u\in C([0,T];L^1_{w;x})$ implies that there exists a sequence $s_n \to s$ such that $u(s_n,x) \to u(s,x)$ for almost
every $(\omega,x)$ and 
\begin{equs}
 \sup_{n\geq 1} \EE\|u(s_n)\|_{L^{m+1}_x}^{m+1} \lesssim 1+ \EE\|\xi\|_{L^{m+1}_x}^{m+1}. 
\end{equs}
Then, again by Fatou's lemma, we have that
\begin{equs}
 \EE\|u(s)\|_{L^{m+1}_x}^{m+1} \lesssim \liminf_{n\to \infty} \EE\|u(s_n)\|_{L^{m+1}_x}^{m+1} \lesssim \sup_{n\geq 1} \EE\|u(s_n)\|_{L^{m+1}_x}^{m+1} 
 \lesssim 1+ \EE\|\xi\|_{L^{m+1}_x}^{m+1},
\end{equs}
which implies that $u(s)\in L^{m+1}_\omega L^{m+1}_x$. We also have the following estimate,
\begin{align*}
 \left|\kappa\int_{s-\frac{1}{\kappa}}^s \int_Q \eta(u) \varphi(s) \varrho \dd x \dd t - \int_Q \eta(u(s)) \varphi(s) \varrho \dd x \right| & = 
 \left|\kappa \int_{s-\frac{1}{\kappa}}^s \int_Q (\eta(u) - \eta(u(s))) \varphi(s) \varrho \dd x \dd t \right| \\
 & \lesssim \|\eta'\|_\infty \left|\kappa\int_{s-\frac{1}{\kappa}}^s \int_Q |u - u(s)| \varphi(s) \varrho \dd x \dd t \right|
\end{align*}
where (passing to a subsequence) the latter term converges to zero, $\PP$-almost surely. This is true since, by Theorem \ref{thm:L^1_cont}, $u$ is continuous in $L^1_\omega L^1_{w;x}$
(thus $\kappa\int_{s-\frac{1}{\kappa}}^s (\EE \int_Q |u - u(s)| w \dd x) \dd t \to 0$), $w$ is strictly positive in $Q$ and $\phi(s)$ has compact support in $Q$. 

Taking $\kappa\to \infty$ in \eqref{eq:entr_sol_s} (passing to a suitable subsequence) implies that $u$ is an entropy solution on $[s,T]$. Last, by uniqueness of entropy solutions and 
the continuity in $L^1_\omega L^1_{w;x}$ we conclude that $u(t)$ coincides with $u_s(t;u(s;\xi))$ in $L^1_\omega L^1_{w;x}$ for every $t\in[s,T]$, which completes the proof. 
\end{proof}

\begin{proof}[Proof of Proposition \ref{prop:Markov}] To prove that the map $P_t$ is a semigroup we need to show that $P_{t+s} = P_s P_t$, for every $0<s\leq t$. The argument follows 
\cite[Proof of Theorem 9.14]{DPZ14}. In particular, we prove that
\begin{equs}
 \EE\left(F(u(t+s;\xi))\big|\mathcal{F}_s\right) = \EE F(u_s(t;\zeta))\big|_{\zeta = u(s;\xi)} \label{eq:markov_1}
\end{equs}
$\PP$-almost surely, for every $\xi\in L^1_{w;x}$ and $F\in C_b(L^1_{w;x})$ (space of bounded continuous functions from $L^{1}_{w,x}$ to $\RR$), which implies the
result by virtue of the monotone class theorem. 

We first consider $\xi\in L^{m+1}_x$. By Corollary \ref{cor:restart} we have that the random variables $u(s+t;\xi)$ and 
$u_s(t;u(s;\xi))$ coincide in $L^1_\omega L^1_{w;x}$. Since $u(s;\xi)$ is $\mathcal{F}_s$-measurable and $u(s;\xi)\in L^{m+1}_\omega L^{m+1}_x$ (by the same argument as in
the proof of Corollary \ref{cor:restart}), it suffices to prove that   
\begin{equs}
 \EE\left(F(u_s(t;\zeta_0))\big|\mathcal{F}_s\right) = \EE F(u_s(t;\zeta))\big|_{\zeta = \zeta_0}, \label{eq:markov_2}
\end{equs}
$\PP$-almost surely, for every $\mathcal{F}_s$-measurable $\zeta_0\in L^{m+1}_\omega L^{m+1}_x$. If $\zeta_0$ is a simple random
variable, \eqref{eq:markov_2} follows easily. Otherwise, there exists a sequence of simple random variables $(\zeta_0^{(n)})_{n\geq 1}$,
such that $\zeta_0^{(n)} \to \zeta_0$ in $L^{m+1}_\omega L^{m+1}_x$ as $n\to \infty$, which in turn implies that $\zeta_0^{(n)} \to \zeta_0$ in 
$L^1_\omega L^1_{w;x}$. By \eqref{eq:markov_2} we know that for every $n\geq 1$
\begin{equs}
 \EE\left(F(u_s(t;\zeta_0^{(n)}))\big|\mathcal{F}_s\right) = \EE F(u_s(t;\zeta))\big|_{\zeta = \zeta_0^{(n)}}. \label{eq:markov_3}
\end{equs}
Using \eqref{eq:v_contr} (with $u$ replaced by $u_s$) we obtain the estimate
\begin{equs}
 \EE\|u_s(t;\zeta_0^{(n)}) - u_s(t;\zeta_0)\|_{L^1_{w;x}} \leq \EE \|\zeta_0^{(n)} - \zeta_0\|_{L^1_{w;x}}
\end{equs}
which implies that $u_s(t;\zeta_0^{(n)}) \to u_s(t;\zeta_0)$ in $L^1_\omega L^1_{w;x}$. Since $F$ is continuous and bounded the 
left hand side of \eqref{eq:markov_3} converges to $\EE\left(F(u_s(t;\zeta_0))\big|\mathcal{F}_s\right)$, $\PP$-almost surely passing to a subsequence.
On the other hand, the mapping $L^1_{w;x}\ni \zeta \mapsto \EE F(u_s(t;\zeta))$ is continuous. Indeed, for arbitrary $\zeta \in L^1_{w;x}$, let $(\zeta_n)_{n\geq 1}$
be a sequence in $L^1_{w;x}$ such that $\zeta_n \to \zeta$ in $L^1_{w;x}$. Then, for every subsequence $k_n\to \infty$ we can use
\eqref{eq:v_contr} (with $u$ replaced by $u_s$) and the continuity and boundedness of $F$ to find a further subsequence $m_{k_n}\to\infty$ such that $\EE F(u_s(t;\zeta_{m_{k_n}})) \to \EE F(u_s(t;\zeta))$.
Using the continuity of $L^1_{w;x}\ni \zeta \mapsto \EE F(u_s(t;\zeta))$ and the fact that $\zeta_0^{(n)} \to \zeta_0$ in $L^1_\omega L^1_{w;x}$, 
we conclude that the right hand side of \eqref{eq:markov_3} converges to $\EE F(u_s(t;\zeta))\big|_{\zeta = \zeta_0}$, $\PP$-almost surely.
Hence, we can the pass to the limit in \eqref{eq:markov_3} to obtain \eqref{eq:markov_2} for arbitrary
$\mathcal{F}_s$-measurable $\zeta_0\in L^{m+1}_\omega L^{m+1}_x$. This proves \eqref{eq:markov_1} for $\xi\in L^{m+1}_{x}$. 

Finally, if $\xi\in L^1_{w,x}$, we take any sequence $\xi_n \to \xi$ in $L^1_{w;x}$ such that $\xi_n\in L^{m+1}_x$. Using the same 
arguments as in the previous paragraph we obtain that $\EE\left(F(u(t+s;\xi_n))\big|\mathcal{F}_s\right) \to \EE\left(F(u(t+s;\xi))\big|\mathcal{F}_s\right)$,
$\PP$-almost surely (upon relabelling a subsequence), and that the mapping $L^1_{w;x}\ni \zeta \mapsto \EE F(u_s(t;\zeta))$ is continuous. This implies \eqref{eq:markov_1}
for $\xi\in L^1_{w,x}$. 

To prove that the semigroup $P_t$ is Feller we simply notice that if $\xi_n \to \xi$ in $L^1_{w;x}$ then, by \eqref{eq:v_contr}, for every subsequence $k_n \to \infty$,
there exists a further subsequence $m_{k_n} \to \infty$ such that $u(t;\xi_{m_{k_n}}) \to u(t;\xi)$ in $L^1_{w;x}$, $\PP$-almost surely. Then, by the continuity and boundedness of $F$,
$P_tF(\xi_{m_{k_n}}) \to P_tF(\xi)$.  
\end{proof}

We are now ready to prove Theorem \ref{thm:pol_mix}, following \cite[Proof of Theorem 4.3.9, Proof of Lemma 4.3.11]{Ro07}.

\begin{proof}[Proof of Theorem \ref{thm:pol_mix}] Step 1: For $\xi\in L^{m+1}_x$ we let $\eta_s(\xi) := u_s(0;\xi)$. By Corollary \ref{cor:restart} for every $s_1\leq s_2\leq -1$ we have that 
\begin{equation*}
 \eta_{s_1}(\xi) = u_{s_2}(0;u_{s_1}(s_2;\xi))
\end{equation*}
in $L^1_\omega L^1_{w;x}$. By Theorem \ref{thm:pol_contr} we know that
\begin{align*}
 \EE \|\eta_{s_2}(\xi) - \eta_{s_1}(\xi)\|_{L^1_{w;x}} \lesssim \|w\|_{L^{m_*}_x}^{m_*} |s_2|^{-\frac{1}{m-1}}.
\end{align*}
The last inequality implies that $(\eta_s(\xi))_{s\leq -1}$ is a Cauchy sequence in $L^1_\omega L^1_{w;x}$. Hence there exists a random variable 
$\eta(\xi)\in L^1_\omega L^1_{w;x}$ such that $\eta_s(\xi) \to \eta(\xi)$ as $s\to-\infty$. 

We claim that $\eta(\xi)$ is independent of $\xi$. Indeed, using again Theorem \ref{thm:pol_contr} we have that for any $\xi,\tilde \xi\in L^{m+1}_x$ 
\begin{align*}
 \EE\|\eta_s(\xi) - \eta_s(\txi)\|_{L^1_{w;x}} \lesssim \|w\|_{L^{m_*}_x}^{m_*} |s|^{-\frac{1}{m-1}} 
\end{align*}
and letting $s\to-\infty$ asserts our claim.  

We now let $\mu = \mathcal{L}(\eta)\in \mathcal{M}_1(L^1_{w;x})$ for $\eta = \eta(0)$. Then $\mu\in \mathcal{M}_1(L^{m+1}_x)$ since 
\begin{align*}
 \EE\|\eta\|_{L^{m+1}_x}^{m+1} \leq \liminf_{s\to-\infty} \EE\|\eta_s(0)\|_{L^{m+1}_x}^{m+1} \leq \sup_{s\leq -1} \EE\|\eta_s(0)\|_{L^{m+1}_x}^{m+1}  
\end{align*}
and the last quantity is bounded by Proposition \ref{prop:L^m+1_bddness}. Similarly to Definition \ref{def:semigroup} we denote by 
$P_{s,t}$ the semigroup associated to \eqref{eq:shifted_spm} at time $t$. In this notation $P_t = P_{0,t}$. Then one has the following
elementary calculation, 
\begin{align*}
 \int_{L^1_{w;x}} P_{0,t}F(\xi) \mu(\mathrm{d} \xi) & = \lim_{s\to\infty} P_{-s,0}(P_{0,t}F)(0)  = \lim_{s\to\infty} P_{-s,t}F(0)
 = \lim_{s\to\infty} P_{-(t+s),0}F(0) = \int_{L^1_{w;x}} F(\xi) \mu(\mathrm{d} \xi)
\end{align*}
for every $F\in C_b(L^1_{w;x})$ (bounded continuous functions from $L^{1}_{w,x}$ to $\RR$), where we also use that $P_{s,t}$ is Feller for 
every $s<t$, as well as the identities $P_{s,t} = P_{s+\tau,t+\tau}$, for every $\tau\in \RR$ and $P_{s,\tau}P_{\tau,t}=P_{s,t}$ for every $s<\tau<t$.
This implies that $P_t^*\mu = \mu$ as measures on $L^1_{w;x}$, hence on $L^{m+1}_x$ since both $P_t^*\mu$ and $\mu$ are supported on $L^{m+1}_x$. 

\smallskip

Step 2: By Theorem \ref{thm:pol_contr}, for every $F\in \mathrm{Lip}(L^1_{w;x})$ and $\xi,\txi\in L^1_{w;x}$ we have that
\begin{align*}
 \left|P_tF(\xi) -P_tF(\txi)\right| \lesssim \|F\|_{\mathrm{Lip}(L^1_{w;x})} \|w\|_{L^{m_*}_x}^{m_*} t^{-\frac{1}{m-1}} 
\end{align*}
which implies that any two invariant measures $\mu$ and $\tilde \mu$ on $L^1_{w;x}$ coincide. Using the last estimate we also see that
\begin{equs}
 \left|P_tF(\xi) - \int_{L^1_{w;x}} F(\txi) \, \mu(\mathrm{d} \txi) \right| \lesssim \|F\|_{\mathrm{Lip}(L^1_{w;x})} \|w\|_{L^{m_*}_x}^{m_*} t^{-\frac{1}{m-1}} 
\end{equs}
which completes the proof by taking the supremum over $\|F\|_{\mathrm{Lip}(L^1_{w;x})}\leq 1$ and $\xi\in L^1_{w;x}$.
\end{proof}

\pdfbookmark[0]{Appendix}{appendix}
\appendix
\appendixpage 
\addtocontents{toc}{\protect\contentsline{section}{Appendix}{\thepage}{appendix.0}}

\stepcounter{section}
\hypertarget{ap:A}{}
\section*{Appendix \Alph{section}}

In this appendix we prove Proposition \ref{prop:viscoous-well-posedenss} using Galerkin approximations as in \cite{DG18}. We also prove Lemma \ref{lem:supess_L^m+1}, which is
used in the proof of Proposition \ref{prop:L^m+1_bddness}.  

\begin{proof}[Proof of Proposition \ref{prop:viscoous-well-posedenss}] We fix $n\geq 1$, and since $n$ is fixed, in order to ease the notation
we drop the dependence on $n$ and we relabel $\bar A:= A_n$, $\bar \sigma= \sigma_n$,  $\bar \xi := \xi_n$, and $u:=u_n$. Let 
$(e_k)_{k=1}^\infty \subset H^4_x\cap H^2_{0,x}$ be an 
orthonormal basis of $L^2_x$ consisting of eigenvectors of $\Delta^2$. For $m=1,2$ let us denote by $H^{-m}_x$ the dual of $H^m_{0;x}$
equipped with the inner product $(\cdot, \cdot)_{H^{-m}_x}:=(\Delta^{-m} \cdot, \Delta^{-m}\cdot)_{H^m_{0;x}}$ and let 
$\Pi_l : H^{-1}_x \to V_l:= \mathrm{span}\{e_1,...,e_l\}$ be the projection operator, that is, for $v \in H^{-1}_x$
\begin{equs}
\Pi_l v:= \sum_{i=1}^l {}_{H^{-1}} \langle v, e_i \rangle_{H^1_{0;x}} e_i.
\end{equs}
Then, the Galerkin approximation 
\begin{equs}          
\begin{cases} \label{eq:Galerkin}
& du_l = \Pi_l  \Delta \bar{A}(u_l)  \dd t 
+ \Pi_l \bar{\sigma}^{k}(u_l) \dd \beta^k(t) 
\\
& u(0) = \Pi_l \bar\xi,
\end{cases}
\end{equs}
is an equation on $V_l$ with locally Lipschitz continuous coefficients having linear growth. Consequently, it admits a unique solution $u_l$, 
for which we have  
\begin{equation*}
u_l \in  L^2_{\omega,t}H^1_{0;x} \cap L^2_\omega C([0,T] ; L^2_x).
\end{equation*}
After applying It\^o's formula for the function $u \mapsto \|u\|_{L^2_x}^2$, we obtain by standard arguments that
\begin{equation} \label{eq:est-un-H1}
 \E \int_0^T \|u_l \|^2_{H^1_{0;x}} \dd t  \leq  C \left(1+\E \|\bar \xi\|_{L^2_x}^2\right),
\end{equation}
and that for every $p\geq 2$
\begin{equation} \label{eq:bound-L2^p}
\E \sup_{t \leq T }\|u_l(t)\|_{L^2_x}^p \leq C \left(1+\E \|\bar \xi\|_{L^2_x}^p\right),
\end{equation}
where $C$ is independent of $l$. In $H^{-1}_x$ we have $\PP$-almost surely, for every $t \in [0,T]$
\begin{align*}
u_l(t)& = \Pi_l \bar\xi + \int_0^t \Pi_l \Delta  \bar{A}(u_l) \dd s
+ \int_0^t \Pi_l \bar{\sigma}^k(u_l) \dd \beta^k(s) =J^1_l+J^2_l(t)+J^3_l(t).
\end{align*}
By Sobolev's embedding theorem and \eqref{eq:est-un-H1} 
\begin{equation*}
 \sup_l \E \|J^2_l\|^2_{W^{\frac{1}{3},4}_tH^{-1}_x} \leq \sup_l \E \|J^2_l \|^2_{H^{1}_t H^{-1}_x}  < \infty.
\end{equation*}
By \cite[Lemma 2.1]{FG95}, the linear growth of $\sigma$ and \eqref{eq:bound-L2^p} we have 
\begin{equation*}
\sup_l \E \|J^3_l\|^p_{W^{\alpha,p}_t H^{-1}_x} < \infty
\end{equation*}
for every $\alpha \in (0, \frac{1}{2})$ and $p \geq 2$. By the last two estimates and \eqref{eq:est-un-H1} we obtain 
\begin{equation*}
\sup_l \E \|u_l\|_{W^{\frac{1}{3},4}_t H^{-1}_x \cap L^2_t H^1_{0;x}} < \infty.
\end{equation*}
By virtue of \cite[Theorems 2.1 and 2.2]{FG95} one can easily see that the embedding
\begin{equs}
W^{\frac{1}{3},4}_t H^{-1}_x \cap L^2_t H^1_{0;x}
\hookrightarrow L^2_t L^2_x \cap C([0,T];H^{-2}_x) =:\mathcal{X}
\end{equs}
is compact. It follows that for any sequences $(l_q)_{q \geq 1}$, $(\bar{l}_q)_{q \geq 1}$, the laws of $(u_{l_q}, u_{\bar{l}_q})_{q\geq 1}$ are 
tight on $\mathcal{X} \times \mathcal{X}$. Let us set 
\begin{equation*}
\beta(t)= \sum_{k=1}^\infty\frac{1}{\sqrt{2^k}}\beta^k(t)\mathfrak{e}_k,
\end{equation*}
where $(\mathfrak{e}_k)_{k=1}^\infty$ is the standard orthonormal basis of $\ell^2$. By Prokhorov's theorem, there exists a (non-relabelled) subsequence
$(u_{l_q}, u_{\bar l_q})_{q\geq 1}$ such that the laws of $(u_{l_q}, u_{\bar l_q}, \beta)_{q\geq 1}$ on $\mathcal{Z}:= \mathcal{X} \times \mathcal{X} \times C([0,T]; \ell^2)$ 
are weakly convergent. By Skorohod's representation theorem, there exist $\mathcal{Z}$-valued random variables $(\hat{u}, \check{u}, \tilde{\beta})$, $(\widehat{u_{l_q}}, \widecheck{u_{\bar l_q}}, \tilde{\beta}_q)$, $q\geq 1$, 
on a probability space $(\tilde{\Omega}, \tilde{\mathcal{F}}, \tilde{\bP})$  such that in $\mathcal{Z}$, $\tilde{\bP}$-almost surely,
\begin{equation} \label{eq:convergence-in-Z}
(\widehat{u_{l_q}}, \widecheck{u_{\bar l_q}}, \tilde{\beta}_q)\to (\hat{u}, \check{u}, \tilde{\beta}),
\end{equation}
as $l \to \infty$, and for each $q \geq 1$, as random variables in $\mathcal{Z}$
\begin{equation} \label{eq:distribution}
 (\widehat{u_{l_q}}, \widecheck{u_{\bar l_q}}, \tilde{\beta}_q)\overset{d}{=}(u_{l_q}, u_{\bar l_q}, \beta).
\end{equation}
Moreover, upon passing to a non-relabelled subsequene, we may assume that 
\begin{equation} \label{eq:almost-everywhere}
 (\widehat{u_{l_q}}, \widecheck{u_{\bar l_q}}) \to (\hat{u}, \check{u}),
\end{equation}
for almost every $(\tilde{\omega}, t,x)$. Let $(\tilde{\mathcal{F}}_t)_{t \in [0,T]}$ be the augmented filtration of $\mathcal{G}_t:= \sigma( \hat{u}(s), \check{u}(s), \tilde \beta(s); s \leq t)$, 
and let $\tilde{\beta}^k(t):= \sqrt{2^k}(\tilde{\beta}(t), \mathfrak{e}_k)_{\ell^2}$. It is easy to see that $\tilde{\beta}^k$, $k \geq 1$, are independent, standard, real-valued $\tilde{\mathcal{F}}_t$-Wiener processes.
Indeed, they are $\tilde{\mathcal{F}}_t$-adapted by definition and they are independent since $\beta^k$ are. We only have to show that they are $\tilde{\mathcal{F}}_t$-Wiener processes. Let us fix $s < t$ and let $V$ be
a bounded continuous function on $C([0,s]; H^{-2}_x) \times C([0,s]; H^{-2}_x) \times C([0, s] ; \ell^2)$. For each $l \geq 1$ we have 
\begin{equs}
\tilde{\E} \left(\tilde{\beta}^k_q(t)-\tilde{\beta}^k_q(s)\right)V\left(\widehat{u_{l_q}}|_{[0,s]}, \widecheck{u_{\bar l_q}}|_{[0,s]}, {\tilde{\beta}_q}|_{[0,s]}\right)
= \E\left(\beta^k(t)-\beta^k(s)\right)V\left( u_{l_q}|_{[0,s]}, u_{\bar l_q}|_{[0,s]}, \beta|_{[0,s]}\right)=0,
\end{equs}
which by using uniform integrability and  passing to the limit $q \to \infty$ shows that $\tilde{\beta}^k(t)$ is a $\mathcal{G}_t$-martingale.  
Similarly, $|\tilde{\beta}^k(t)|^2-t$  is a $\mathcal{G}_t$-martingale. By continuity of $\tilde{\beta}^k(t)$ and $|\beta^k(t)|^2-t$, and the fact that
their supremum in time is integrable in $\omega$, one can easily see that they are also $\tilde{\mathcal{F}}_t$-martingales. Hence, by L\'evy's 
characterisation theorem (see, e.g., \cite[Theorem 3.16]{KS91}) $\tilde{\beta}^k$ are $\tilde{\mathcal{F}}_t$-Wiener processes.

We now show that $\hat{u}$ and $\check{u}$ both satisfy the equation 
\begin{equs}
 dv & = \Delta \bar{A}(v) \, dt 
 + \bar{\sigma}^{k}(x,v)\dd \tilde{\beta}^k(t).
\end{equs}
Notice that due to \eqref{eq:est-un-H1}, we have that $\hat{u} \in L^2_{\tilde{\omega},t} H^1_{0;x}$. Let us set
\begin{equs}
\nonumber
\hat{M}(t) &:= \hat{u}(t)- \hat{u}(0)-\int_0^t \Delta \bar{A}(\hat u) \dd s
\\
\hat{M}_q(t) &:= \widehat{u_{l_q}}(t)- \widehat{u_{l_q}}(0)-\int_0^t  \Pi_{l_q}\Delta \bar{A}(\widehat{u_{l_q}}) \dd s
\\
M_q(t) &:= u_{l_q}(t)- u_{l_q}(0)-\int_0^t   \Pi_{l_q}\Delta \bar{A}(u_{l_q}) \dd s.
\end{equs}
We will show that for any $\phi \in H^{-2}_x$ and $k \geq 1$, the processes 
\begin{equs}
 \hat M^1(t) &:= (\hat M(t), \phi)_{H^{-2}},
 \\
 \hat M^2(t) &:= (\hat M(t), \phi)^2_{H^{-2}}-\int_0^t \sum_{k=1}^\infty | (\bar{\sigma}^{k}(\hat{u}), \phi)_{H^{-2}}|^2 \dd s,
\end{equs}
and 
\begin{equation*}
\hat M^{3,k}(t):=\tilde{\beta}^k(t)(\hat M(t), \phi)_{H^{-2}_x}- \int_0^t (\bar{\sigma}^{k}(\hat{u}), \phi)_{H^{-2}} \dd s
\end{equation*}
are continuous $\tilde{\mathcal{F}}_t$-martingales. We first show that they are continuous $\mathcal{G}_t$-martingales. Assume for 
now that $\phi = \Delta^2\psi$, where $\psi \in V_{l_q}$. For, $i=1,2,3$, let us also define the processes $\hat M^i_q, M^i_q$ similarly 
to $\hat M^i$, but with $\hat M$, $\hat u$, $\bar{\sigma}^{k}(\cdot)$ replaced by $\hat M_q, \widehat{u_{l_q}}$, $\Pi_{l_q}\bar{\sigma}^{k}(\cdot)$ and $M_q, u_{l_q}$, $\Pi_{l_q} \sigma^{k}(\cdot)$, 
respectively. Let us fix $s < t$ and let $V$ be a bounded continuous function on $C([0,s]; H^{-2}_x) \times C([0, s] ; \ell^2)$. We have that 
\begin{equs}
 (M_q(t), \phi)_{H^{-2}}&= \int_0^t (\Pi_{l_q}\bar{\sigma}^{k}(u_{l_q}), \phi)_{H^{-2}_x} \dd \beta^k(s).
\end{equs}
It follows that $M^i_q$ are continuous $\mathcal{F}_t$-martingales. Hence, 
\begin{equation*}
 \E V\left(u_{l_q}|_{[0,s]},u_{\bar l_q}|_{[0,s]}, \beta|_{[0,s]}\right)\left(M_q^i(t)-M_q^i(s)\right)=0,
\end{equation*}
which combined with \eqref{eq:distribution} gives 
\begin{equation} \label{eq:martingale-Ml}
 \tilde{\E}V\left(\widehat{u_{l_q}}|_{[0,s]}, \widecheck{u_{\bar l_q}}|_{[0,s]},{\tilde{\beta}_q}|_{[0,s]}\right) \left(\hat M_q^i(t)-\hat M_q^i(s)\right) =0.
\end{equation}
Next, notice that 
\begin{equs}
 \tilde{\E} \int_0^T\left| \left( \Pi_{l_q}\Delta \bar{A}(\widehat{u_{l_q}})- \Delta \bar{A}(\hat u), \phi \right)_{H^{-2}_x} \right| \dd t 
 & =  \tilde{\E} \int_0^T \left| \left(\bar{A}(\widehat{u_{l_q}})- \bar{A}(\hat{u}), \Delta \psi \right)_{L^2_x} \, \right| \dd t  
 \\
 & \leq  \tilde{\E} \int_0^T \| \hat{u}- \widehat{u_{l_q}}\|_{L^2_x} \to 0, \label{eq:conPhi}
\end{equs}
where the convergence follows from  \eqref{eq:convergence-in-Z} and the fact that $\left(\int_0^T \|\widehat{u_{l_q}}-\hat{u}\|_{L^2_x}\dd t\right)_{q\geq 1}$ is uniformly integrable 
on $\Omega$ (which in turn follows from \eqref{eq:est-un-H1}). Hence, by \eqref{eq:conPhi} and \eqref{eq:convergence-in-Z} we see that for each $t \in [0,T]$
\begin{equation} \label{eq:MltoM-in-probability}
 (\hat M_q(t), \phi)_{H^{-2}_x} \to (\hat M(t), \phi)_{H^{-2}_x}
\end{equation} 
in probability. Then, one can easily verify that $\hat{M}^i_q(t) \to \hat M^i(t)$ in probability. Moreover, for every $\phi \in H^{-2}_x$ and 
every $p \geq 2$ we have, by \eqref{eq:distribution} and \eqref{eq:bound-L2^p},
\begin{equs}
\sup_q \tilde{\E}| (\hat{M}_q(t), \phi)_{H^{-2}_x}|^p&= \sup_q \E \left|   \int_0^t (\Pi_{l_q} \bar{\sigma}^{k}(u_{l_q}), \phi)_{H^{-2}_x} \dd \beta^k(s)\right|^p
\\
&\le \|\phi\|_{H^{-2}_x}^p \E( 1+ \|\bar\xi\|_{L^2_x}^p).
\end{equs}
From this, one easily deduces that for each $i=1,2,3$, and  $t\in [0,T]$, $M^i_q(t)$ are uniformly integrable. Hence, we can pass to the limit in
\eqref{eq:martingale-Ml} to obtain
\begin{equation} \label{eq:martingale-property}
\tilde{\E}V\left(\hat u|_{[0,s]},\check u|_{[0,s]} {\tilde{\beta}}|_{[0,s]}\right)\left(\hat M^i(t)-\hat M^i(s)\right) =0.
\end{equation}
In addition, using the continuity of $\hat M^i(t)$ in $\phi$, uniform integrability, and the fact that $\cup_q \Delta^2 V_{l_q}$ is dense in $H^{-2}_x$, it follows that
\eqref{eq:martingale-property} holds also for every $\phi \in H^{-2}_x$. Hence, for every $\phi \in H^{-2}_x$, $\hat{M}^i$ are continuous $\mathcal{G}_t$-martingales having all
moments finite. In particular, by Doob's maximal inequality, they are uniformly integrable (in $t$), which combined with continuity (in $t$) implies that they are also $\tilde{\mathcal{F}}_t$-martingales.
By \cite[Proposition A.1]{Hof13} we obtain that $\tilde{\PP}$-almost surely, for every $\phi \in H^{-2}_x$ and $t \in [0,T]$.
\begin{equs}
(\hat{u}(t), \phi)_{H^{-2}}&= (\hat{u}(0), \phi)_{H^{-2}_x} 
+\int_0^t (\Delta \bar{A}(\hat{u}), \phi)_{H^{-2}_x}\dd s + \int_0^t (\bar{\sigma}^{k}(\hat u), \phi)_{H^{-2}_x} \dd \tilde{\beta}^k(s).
\end{equs}
Notice that $\hat{u}(0)\overset{d}{=} \bar \xi$, which implies that $\hat{u}(0) \in L^{m+1}_x$ $\tilde{\PP}$-almost surely.
Choosing $\phi =\Delta^2 \psi$ for $\psi \in C_c^\infty(Q)$, we obtain that for almost every $(\tilde{\omega}, t)$
\begin{equs}
 (\hat{u}(t), \psi)_{L^2_x}&= (\hat{u}(0), \psi)_{L^2_x}-\int_0^t  \left(\nabla \bar{A}(\hat{u}), \nabla  \psi\right)_{L^2_x} \dd s
 + \int_0^t ( \bar{\sigma}^{k}(\hat u) ,  \psi)_{L^2_x} \dd \tilde{\beta}^k(s).
\end{equs}
It follows (see \cite{KR79}) that $\hat{u}$ is a continuous $L^2_x$-valued $\tilde{\mathcal{F}}_t$-adapted process. Hence, $\hat u$ is an $L^2_x$-solution
of $\mathcal{E}(\bar A, \hat \sigma, \hat \xi )$ on $(\tilde{\Omega}, (\tilde{\mathcal{F}}_t)_t , \tilde{\bP})$ with driving noise $(\tilde{\beta}^k)_{k=1}^\infty$, where $\hat{\xi}:= \hat{u} (0)$. 
Again, by standard arguments, for every $p \geq 2$ one has the estimate
\begin{equs}
\E \sup_{t \leq T }\|\hat u(t)\|_{L^p_x}^p+\E \int_0^T \int_{Q}|\hat u|^{p-2} |\nabla \hat u|^2 \dd x \dd t  \leq N (1+\E \|\bar \xi\|_{L^2_x}^p).
\end{equs}
Using this, It\^o's formula (see, e.g., \cite{Kry13}) for the function $u \mapsto \int_Q \eta(u)\varrho \, dx$, and It\^o's product rule, one can see that $\hat u$ is 
an entropy solution of $\mathcal{E}(\bar A, \hat \sigma, \hat \xi )$ on $(\tilde{\Omega}, (\tilde{\mathcal{F}}_t)_t , \tilde{\bP})$ with driving noise $(\tilde{\beta}^k)_{k=1}^\infty$ and initial condition $\hat{\xi}:= \hat{u} (0)$.  
In the exact same way $\check{u}$ is an $L^2_x$-solution and an entropy solution of $\mathcal{E}( \bar A, \bar \sigma,  \check \xi)$ (again, on $(\tilde{\Omega}, (\tilde{\mathcal{F}}_t)_t , \tilde{\bP})$) 
with driving noise $(\tilde{\beta}^k)_{k=1}^\infty$ and $\check \xi:= \check u (0)$. Furthermore, we have for $\delta>0$
\begin{equs}
\tilde \bP(\| \hat \xi - \check \xi \|_{H^{-2}_x}> \delta) & \leq \delta^{-1} \tilde \E \|\hat \xi - \check \xi \|_{H^{-2}_x}
\\
& \leq \liminf_{q \to \infty} \delta^{-1} \tilde \E \|\widehat{u_{l_q}} (0) - \widecheck{u_{\bar l_q}}(0)\|_{H^{-2}_x}
\\
&= \liminf_{q \to \infty} \delta^{-1}  \E \|\Pi_{l_q}\bar \xi - \Pi_{\bar l_q}\bar \xi\|_{H^{-2}_x} =0.
\end{equs}
Hence $\hat{u}$ and $\check{u}$ are both entropy solutions with the same initial condition. Moreover, by Lemma \ref{lem:strong entropy} they have the $(\star)$-property \eqref{eq:star_prop}. Hence, by Corollary \ref{eq:contraction}
we conclude that $\hat{u}= \check{u}$. By \cite[Lemma 1.1]{GK96} we have that the initial sequence $(u_l)_{l=1}^\infty$ converges in probability in $\mathcal{X}$ to some $u \in \mathcal{X}$. Using this convergence and the uniform 
estimates on $u_l$, it  is then straightforward to pass to the limit in \eqref{eq:Galerkin} and see that $u$ is indeed an $L^2_x$-solution.
\end{proof}

\begin{lemma} \label{lem:supess_L^m+1} Let $u$ be an entropy solution of $\mathcal{E}(A, \sigma, \xi)$ for $A(r) = |r|^{m-1}r$. There exists $C>0$ depending only
on $d$, $k$, $m$, and $|Q|$, but not on $T$ and $\xi$ such that 
\begin{equs}
 \supess_{t \in [0,T]} (t\wedge1)^{\frac{m+1}{m-1}} \E \|u(t)\|_{L^{m+1}_x}^{m+1} \leq C.
\end{equs}
%
\end{lemma}

\begin{proof}
Let $u_n$ be the unique entropy solution of $\mathcal{E}(A_n, \sigma, \xi)$, where $A_n(r)=|r|^{m-1}r+n^{-1}r$. It is easy to see by construction that $u_n$ is an $L^2_x$-solution (see Definition \ref{def:L^2_sol}). 
By It\^o's formula we see that
\begin{equs}
 \E\|u_n(t)\|_{L^{m+1}_x}^{m+1} & = \E\|\xi\|_{L^{m+1}_x}^{m+1} -\E \int_0^t\left((m+1)\| \nabla u^m_n\|_{L^2_{x}}^2 + \frac{(m+1)m}{2n}\| \nabla u_n^{\frac{m+1}{2}}\|_{L^2_{x}}^2\right) \dd s  
 \\
 & \quad + \E \int_0^t\int_Q\frac{(m+1)m}{2} |u_n|^{m-1} |\sigma(u_n)|_{\ell^2}^2 \dd x \dd s
\end{equs}
By Poincare's inequality and Jensen's inequality we have
\begin{equs}
 \left( \E \|u_n\|_{L^{m+1}_x}^{m+1}\right)^{\frac{2m}{m+1}} \lesssim \| \nabla u^m_n\|_{L^2_{x}}^2.
\end{equs}
From this we obtain for $f(t):= \E\|u_n(t)\|_{L^{m+1}_x}^{m+1}$ the differential inequality
\begin{equs}
 f'(t) \leq -C_1 |f(t)|^{\frac{2m}{m+1}}+C_2f(t)+C_3,
\end{equs}
where for $i=1,..,3$, $C_i$ is a positive constant which depends on $d, K, m, |Q|$ but not on $T$. Using Young's inequality we obtain that
\begin{equs}
 f(t) \leq C \left(-|f(t)|^{\frac{2m}{m+1}} +1\right) 
\end{equs}
and by a simple comparison criterion (see \cite[Lemma 3.8]{TW18})
\begin{equs}
 f(t) \leq C (t\wedge 1)^{-\frac{m+1}{m-1}} 
\end{equs}
for some $C>0$ independent of $T$ and $f(0)$. The result follows if we let $n\to \infty$, since $u_n \to u$ in 
$L^1_{\omega,t}L^1_{w;x}$.
\end{proof}

\stepcounter{section}
\section*{Appendix \Alph{section}}

\begin{lemma} \label{lem:lwr_bd} For every $m\geq 1$ there exists $C>0$ such that
\begin{equation*} 
 \left|\left(|u|^{m-1} u - |v|^{m-1}v\right)\right| \geq C |u-v|^m,
\end{equation*}
for every $u,v\in \RR$.
\end{lemma}

\begin{proof} Let $f(z) = |z|^{m-1} z$, $z\in \RR$. Without loss of generality we can assume that $u-v>0$, which by 
monotonicity implies that $f(u) - f(v) = |u|^{m-1} u - |v|^{m-1}v$. We write
\begin{equation*}
 f(u) - f(v) = \int_v^u \frac{\dd}{\dd z} f(z) \dd z = m \int_v^u |z|^{m-1} \dd z.
\end{equation*}
We distinguish among the cases $u,v\geq 0$, $u,v\leq 0$ and $u\geq 0\geq v$.

\smallskip

Case $u,v\geq 0$: By a change of variables and the monotonicity of $\frac{\dd}{\dd z} f$ on $[0,\infty)$
we have that
\begin{align*}
 f(u) - f(v) = \int_0^{u-v} \frac{\dd}{\dd z}f(z+v) \dd z \geq \int_0^{u-v} \frac{\dd}{\dd z}f(z) \dd z
 = (u-v)^m.
\end{align*}

\smallskip

Case $u,v\leq 0$: By symmetry we get  
\begin{equation*}
 f(u) - f(v) = \int_{-u}^{-v} \frac{\dd}{\dd z}f(z) \dd z.
\end{equation*}
Then, similarly to the first case, we observe that
\begin{equation*}
 f(u) - f(v) = \int_0^{-(v-u)} \frac{\dd}{\dd z}f(z - u) \dd z \geq \int_0^{-(v-u)} f(z) \dd z = (u-v)^m.
\end{equation*}

\smallskip

Case $u\geq 0\geq v$: We first assume that $-\frac{1}{2} (u-v) \leq v \leq 0$. In this case 
$u\geq \frac{1}{2} (u-v)$ and we get the bound
\begin{equation*}
 f(u) - f(v) \geq \int_0^{\frac{1}{2}(u-v)} \frac{\dd}{\dd z} f(z) \dd z = \frac{1}{2^m} (u-v)^m. 
\end{equation*}
We now assume that $v\leq -\frac{1}{2} (u-v)$. Then
\begin{equation*}
 f(u) - f(v) \geq \int_{-\frac{1}{2}(u-v)}^0 \frac{\dd}{\dd z} f(z) \dd z = \frac{1}{2^m} (u-v)^m.
\end{equation*}

The above assertions complete the proof.
\end{proof}

\begin{lemma} \label{lem:int_ineq} Let $f,h:\RR \to \RR$ be continuous functions such that for every $s\leq t$
 \begin{align*}
  \begin{cases}
   & f(t) - f(s) \leq - C \int_s^t |f(\tau)|^{m-1} f(\tau) \dd \tau \\
   & h(t) - h(s) = - C \int_s^t |h(\tau)|^{m-1} h(\tau) \dd \tau 
  \end{cases}
 \end{align*}
 for some $C>0$. Then, provided $f(0) \leq h(0)$, we have that $f(t) \leq h(t)$ for every $t\geq 0$.
\end{lemma}

\begin{proof} Let $h_\eps(t)$ be the unique continuous solution of 
\begin{equation*}
 \begin{cases}
  & h_\eps'(t) = - C |h_\eps(t)|^{m-1} h_\eps(t) + \eps \\
  & h_\eps(0) = h(0).
 \end{cases}
\end{equation*}
We show that for every $t>0$ and $\eps\in(0,1)$ we have that $f(t) \leq h_\eps(t)$. The result is then immediate if we pass to the limit $\eps\to 0$. 

Assume for contradiction that there exists $t>0$ such that $f(t) > h_\eps(t)$. Using the continuity of $f$ and $h_\eps$ and the fact that $f(0) \leq h(0)$
we may assume that there exist $s <  t_* \leq t$ such that $f(s) = h_\eps(s)$ and $f(\tau) > h_\eps(\tau)$ for every $\tau\in (s,t_*]$. This implies that for every 
$\tau\in(s,t_*]$
\begin{align*}
 h_\eps(\tau) - h_\eps(s) < f(\tau) - f(s) \leq - C \int_s^\tau |f(\bar \tau)|^{m-1} f(\bar \tau) \dd \bar \tau. 
\end{align*}
Multiplying the last inequality by $(\tau-s)^{-1}$ and passing to the limit $\tau \to s$ we obtain that
\begin{align*}
 -C |h_\eps(s))|^{m-1} h_\eps(s) + \eps = h_\eps'(s) \leq - C |f(s)|^{m-1} f(s)
\end{align*}
which is a contradiction since $f(s) = h_\eps(s)$. This completes the proof.
\end{proof}

\pdfbookmark{References}{references}
\addtocontents{toc}{\protect\contentsline{section}{References}{\thepage}{references.0}}

\bibliographystyle{alpha}
\bibliography{porous_media_bibliography}{}

\begin{flushleft}
\small \normalfont
\textsc{Konstantinos Dareiotis\\
Max--Planck--Institut f\"ur Mathematik in den Naturwissenschaften\\
04103 Leipzig, Germany}\\
\texttt{\textbf{konstantinos.dareiotis@mis.mpg.de}}
\end{flushleft}

\begin{flushleft}
\small \normalfont
\textsc{Benjamin Gess\\
Max--Planck--Institut f\"ur Mathematik in den Naturwissenschaften\\ 
04103 Leipzig, Germany\\
Faculty of Mathematics, University of Bielefeld\\
33615 Bielefeld, Germany}\\
\texttt{\textbf{benjamin.gess@mis.mpg.de}}
\end{flushleft}

\begin{flushleft}
\small \normalfont
\textsc{Pavlos Tsatsoulis\\
Max--Planck--Institut f\"ur Mathematik in den Naturwissenschaften\\
04103 Leipzig, Germany}\\
\texttt{\textbf{pavlos.tsatsoulis@mis.mpg.de}}
\end{flushleft}

\end{document}